\documentclass[12pt]{amsart}
\usepackage{amsmath,amssymb,amsxtra,anysize,combelow,tikz}

\usepackage{mathrsfs}
\usepackage{times}
\usepackage{anysize}
\usepackage{graphicx,underscore}

\usepackage{tikz}
\usepackage{tikz-cd}

\input xy
\xyoption{all}
\usepackage[all,knot]{xy}
\usepackage{youngtab}
\usepackage{young}
\usepackage{algorithmic}

\marginsize{1in}{1.5in}{1in}{1.5in}

\usepackage{todonotes}

\numberwithin{equation}{subsection}

\theoremstyle{plain}
\newtheorem{theorem}[equation]{Theorem}
\newtheorem{lemma}[equation]{Lemma}
\newtheorem{proposition}[equation]{Proposition}

\theoremstyle{definition}
\newtheorem{definition}[equation]{Definition}
\newtheorem{example}[equation]{Example}
\newtheorem{remark}[equation]{Remark}

\DeclareMathOperator{\SD}{SD}
\DeclareMathOperator{\ch}{ch}
\DeclareMathOperator{\Hilb}{Hilb}
\DeclareMathOperator{\PFH}{PFH}

\DeclareMathOperator{\End}{End}
\DeclareMathOperator{\sort}{sort}
\DeclareMathOperator{\Res}{Res}
\DeclareMathOperator{\Ker}{Ker}
\DeclareMathOperator{\pExp}{Exp}

\DeclareMathOperator{\Id}{Id}

\DeclareMathOperator{\terms}{terms}
\DeclareMathOperator{\LT}{LT}

\newcommand{\BC}{\mathbb{C}}
\newcommand{\BT}{\mathbb{T}}

\newcommand{\BZ}{\mathbb{Z}}
\newcommand{\BA}{\mathbb{A}}
\newcommand{\BB}{\mathbb{B}}
\newcommand{\BQ}{\mathbb{Q}}

\newcommand{\CL}{\mathcal{L}}
\newcommand{\CN}{\mathcal{N}}
\newcommand{\sq}{\square}

\author{Erik Carlsson}
\address{Department of Mathematics, University of California\newline One Shields Avenue, Davis CA 95616}
\email{ecarlsson@math.ucdavis.edu}
\author{Eugene Gorsky}
\address{Department of Mathematics, University of California\newline One Shields Avenue, Davis CA 95616}
 \address{National Research University Higher School of Economics\newline Usacheva 6, Moscow, Russia}
 \email{egorskiy@math.ucdavis.edu}
\author{Anton Mellit}
\address{Faculty of Mathematics, University of Vienna\newline Oskar-Morgenstern-Platz 1, 1090 Wien, Austria}
\email{anton.mellit@univie.ac.at}
\title{The $\BA_{q,t}$ algebra and parabolic flag Hilbert schemes}

\begin{document}
\maketitle

\begin{abstract}
The earlier work of the first and the third named authors introduced the algebra $\BA_{q,t}$ and its polynomial representation.
In this paper we construct an action of this algebra on the equivariant K-theory of certain smooth strata in the flag Hilbert schemes of points on the plane. In this presentation, the fixed points of torus action correspond to generalized Macdonald polynomials and the 
the matrix elements of the operators have explicit combinatorial presentation. 
\end{abstract}

\section{Introduction}

In the earlier article the first and the third named authors \cite{carlsson2015proof} introduced a new and interesting algebra called the algebra $\BA_{q,t}$. It acts on the space $V=\bigoplus_{k=0}^{\infty}V_k$, where $V_k=\Lambda\otimes \BC[y_1,\ldots,y_k]$ and $\Lambda$ is the ring of symmetric functions in infinitely many variables. The algebra has generators $y_i,z_i,T_i,d_+$ and  $d_-$. On each subspace $V_k$,
$y_i$ act as multiplication operators, $T_i$ as Demazure-Lusztig operators, so together they form an affine Hecke algebra. The operators $z_i$ and $T_i$ also form an affine Hecke algebra (in particular, $z_i$ commute). Finally, the most interesting operators 
$d_+:V_{k}\to V_{k+1}$ and $d_{-}:V_{k}\to V_{k-1}$ intertwine different subspaces. 

The algebra $\BA_{q,t}$ was used in \cite{carlsson2015proof} to prove a long-standing {\em Shuffle Conjecture} in algebraic combinatorics \cite{haglund2005combinatorial}. Later, it was also used in \cite{mellit2016toric}  to prove a ``rational" version of Shuffle conjecture introduced in \cite{GN}. The latter yields a combinatorial expression for certain matrix elements of the generators $P_{m,n}$ of the elliptic Hall algebra \cite{SV} acting in its polynomial representation. In particular, the operator $P_{m,n}:V_0\to V_0$ was realized in \cite{mellit2016toric} inside the algebra $\BA_{q,t}$.

It is known from the work of Schiffmann, Vasserot \cite{SV}, Feigin, Tsymbaliuk \cite{FT} and Negu\cb{t} \cite{negut2012moduli} that the elliptic Hall algebra acts on the equivariant $K$-theory of the Hilbert schemes of points on the plane. In particular, \cite{negut2012moduli}
realized $P_{m,n}$ by an explicit geometric correspondence. This leads to a natural question: is there a geometric interpretation of the algebra $\BA_{q,t}$ and its representation $V_\bullet$? We answer this question in the present paper.

The key geometric object is the {\em parabolic flag Hilbert scheme} $\PFH_{n,n-k}$ which is defined as the moduli space of flags $\{I_{n-k}\supset \ldots \supset I_{n}\}$, where $I_s$ are ideals in $\BC[x,y]$ of codimension $s$ and $yI_{n-k}\subset I_n$.  We prove that this is in fact a smooth quasiprojective variety. The following theorem is the main result of the paper. 

\begin{theorem}
Let $U_k=\bigoplus_{n=k}^{\infty}K_{\BC^*\times \BC^*}(\PFH_{n,n-k})$ and let $U_\bullet=\bigoplus_{k=0}^{\infty}U_k$.
Then there is an action of the algebra $\BA_{q,t}$ on $U_\bullet$ and isomorphisms $U_k\simeq V_k$ for all $k$ compatible with the $\BA_{q,t}$-algebra action.
\end{theorem}

The construction of the action of the generators of $\BA_{q,t}$ is quite natural. The action of $z_i$ and $T_i$ follows the classical work of Lusztig on the action of affine Hecke algebras on flag varieties \cite{lusztig1985equivariant}. In particular, $z_i$ correspond to natural line bundles $\CL_i=I_{n-i-1}/I_{n-i}$ on $\PFH_{n,n-k}$. The operators $d_{\pm}$ change the length of the flag and correspond to natural projections $\PFH_{n+1,n-k}\to \PFH_{n,n-k}$ and $\PFH_{n,n-k}\to \PFH_{n,n-k+1}$. Finally, the operators $y_i$ can be obtained using the commutation relations between $d_+,d_-$ and $T_i$.

We compare this geometric construction with \cite{FT,SV,negut2012moduli}.  The key operator in \cite{FT,SV} is realized by a simple Nakajima correspondence $\Hilb^{n,n+1}$ with some power $\CL^k$ of a line bundle on it,  which naturally projects to $\Hilb^n$ and $\Hilb^{n+1}$. This yields an operator $P_{1,k}:K(\Hilb^n)\to K(\Hilb^{n+1})$. We regard $\Hilb^{n,n+1}$ as a cousin of $\PFH_{n+1,n}$, and decompose $P_{1,k}$ as a composition of three operators $P_{1,k}=d_{-}z_1^{k}d_{+}$. Here $d_{+}:U_0\to U_1$ and $d_{-}:U_1\to U_0$ correspond to the pullback and the pushforward under projections, and $z_1:U_1\to U_1$ corresponds to the line bundle $\CL$. We make a similar comparison with the construction of \cite{negut2012moduli} for more complicated operators $P_{m,n}$ in the elliptic Hall algebra.

A combinatorial consequence of this work is the construction of generalized Macdonald basis corresponding to the fixed points of the torus action in $\PFH_{n,n-k}$. For $k=0$ we recover the modified Macdonald basis corresponding to the fixed points on the Hilbert scheme of points \cite{haiman2002combinatorics}. We explicitly compute the matrix elements for all the generators of $\BA_{q,t}$ in this basis.
In fact, we prove that these new elements have a triangularity property with respect to a version of the Bruhat order for affine
permutations, generalizing the triangularity in the dominance order for usual Macdonald polynomials.

Finally, we would like to outline some future directions. First, the construction of the spaces $\PFH_{n,n-k}$ is very similar to the construction of so-called affine Laumon spaces \cite{feigin2011yangians}. Tsymbaliuk \cite{tsymbaliuk2010quantum} constructed an action of the quantum toroidal algebras $\ddot{U}(gl_k)$ on the $K$-theory of Laumon spaces. In particular, for $k=1$ this action coincides with the action of the elliptic Hall algebra (which is known to be isomorphic to $\ddot{U}(gl_1)$) on the $K$-theory of the Hilbert scheme of points. However, it appears that for $k>1$ his representation is larger than $U_{k-1}$. We plan to investigate the relations between $\BA_{q,t}$ and quantum toroidal algebras in the future.

Second, the results of \cite{GN,GNR,mellit2017homology} suggest a deep relation between Hilbert schemes and elliptic Hall algebra, and categorical link invariants such as Khovanov-Rozansky homology. In particular, a precise relation between the Khovanov-Rozansky homology of $(m,n)$ torus knots and the operators $P_{m,n}$ was proved for $m=n+1$ by Hogancamp \cite{hogancamp2017khovanov} and for general coprime $(m,n)$ by the third author in \cite{mellit2017homology}. It is expected \cite{mellit2016toric} that $\BA_{q,t}$ can be realized as the skein algebra of certain more general tangles in the thickened torus, so it would be interesting to extend the approach of \cite{GNR} to this more general framework.

\section{Acknowledgments}

The authors would like to thank Mikhail Bershtein, Andrei Negu\cb{t} and Monica Vazirani for the useful discussions. The work of E. G. was partially supported by the NSF grants DMS-1559338 and DMS-1700814, Russian Academic Excellence Project 5-100 and RSF-16-11-10160. The work of A. M. was supported by the Advanced Grant ``Arithmetic and Physics of Higgs moduli spaces'' No. 320593 of the European Research Council and by the START-Project Y963-N35 of the Austrian Science Fund.

 \section{The algebra}
 \subsection{$\BA_{q}$}\label{section BAq}
The algebras under consideration can be viewed as path algebras of quivers with vertex set $\BZ_{\geq 0}$ \footnote{A categorically inclined reader can view our algebras as categories with object set $\BZ_{\geq 0}$. Then a representation of a category is a simply a functor to the category of vector spaces.}. So we implicitly assume that all our algebras contain orthogonal idempotents $\Id_i$ ($i\in\BZ_{\geq 0}$) and when we speak of an element $R:i\to j$ for $i,j\in \BZ_{\geq 0}$ we impose relation $R = R \Id_i = \Id_j R$. When we have a representation $V$ of such an algebra we always assume that $V=\bigoplus_{i=0}^\infty V_i$ where $V_i = \Id_i V$. Then any element $R:i\to j$ as above induces a linear map $V_i\to V_j$. To stress the direct sum decomposition above we denote such a representation by $V_\bullet$.

 First we define the ``half algebra'' $\BA_q$ depending on one parameter $q\in\BQ(q)$:
\begin{definition} $\BA_q$ is the $\BQ(q)$-linear algebra generated by a collection of orthogonal idempotents labeled by $\BZ_{\geq 0}$ and elements
\[
d_+:k\to k+1,\; d_-:k\to k-1,\; T_i:k\to k \quad(1\leq i<k),\; y_i:k\to k \quad(1\leq i\leq k)
\]
subject to relations
\begin{equation}\label{eq:Trel}
(T_i-1)(T_i+q)=0, \quad T_i T_{i+1} T_i = T_{i+1} T_i T_{i+1}, \quad T_i T_j = T_j T_i \quad (|i-j|>1),
\end{equation}
\begin{equation}\label{eq:Tyrel}
T_i y_{i+1} T_i = q y_{i} \; (1\leq i\leq k-1),
\end{equation}
\begin{equation*}
y_i T_j = T_j y_i  \; (i\notin \{j, j+1\}),\;
y_i y_j = y_j y_i \; (1\leq i,j \leq k),
\end{equation*}
\begin{equation}\label{eq:Tdminusrel}
d_-^2 T_{k-1} = d_-^2,\;  d_- T_i = T_i d_-\; (1\leq i \leq k-2),\;
d_- y_i = y_i d_-\; (1\leq i\leq k-1),
\end{equation}
\begin{equation}\label{eq:Tdplusrel}
T_1 d_+^2 = d_+^2,\; d_+ T_i = T_{i+1} d_+ \; (1\leq i \leq k-1),
\end{equation}
\begin{equation*}
d_+ y_i = T_1 T_2 \cdots T_i y_i T_i^{-1} \cdots T_1^{-1} d_+,\; (1\leq i\leq k)
\end{equation*}
\begin{equation}\label{eq:commrel}
d_+ d_- - d_- d_+ = (q-1) T_1 T_2 \cdots T_{k-1} y_k.
\end{equation}
\end{definition}

\begin{remark}\label{rem:heckecat}
Note that relations \eqref{eq:Trel} define the Hecke algebra, and relations \eqref{eq:Trel} + \eqref{eq:Tyrel} define the affine Hecke algebra. %One can think of $d_-$ as the operator that ``kills'' the index $k$. Similarly $d_+$ ``creates'' a new index numbered $1$ shifting all the preexisting indices by $1$.
\end{remark} 

In what follows we will need a slightly different description of the algebra $\BA_{q}$. Let the $AH_k$ be the affine Hecke algebra generated by $T_1,\ldots,T_{k-1},y_1,\ldots,y_k$ modulo relations \eqref{eq:Trel} and \eqref{eq:Tyrel}. The following lemma gives another presentation of the algebra $AH_k$ similar to the Iwahori-Matsumoto presentation of the affine Hecke algebra, although in our definition $y_i$ are not invertible. The proof is similar to \cite[Lemma 5.4]{carlsson2015proof}, but we present it here for completeness.

\begin{lemma}
\label{lem:Iwahori}
Consider the algebra $AH'_k$ generated by $T_1,\ldots,T_{k-1}$ and an element $\varphi$ modulo relations \eqref{eq:Trel} and
\begin{equation}
\label{eq:phirel}
\varphi T_i=T_{i+1}\varphi\ (i\le k-2),\ \varphi^2 T_{k-1}=T_1\varphi^2.
\end{equation}
Then the algebras $AH_k$ and $AH'_k$ are isomorphic.
\end{lemma}

\begin{proof}
Define $\varphi=T_1  \cdots T_{k-1} y_k.$
Let us prove that \eqref{eq:Tyrel} imply \eqref{eq:phirel}. For $i\le k-2$ one has:
\[
\varphi T_i=T_1 \cdots T_{k-1} y_kT_i=T_1  \cdots T_{k-1} T_iy_k=T_{i+1}T_1  \cdots T_{k-1} y_k=T_{i+1}\varphi,
\]
while 
\[
\varphi^2 T_{k-1}= T_1 \cdots T_{k-1}y_kT_1 \cdots T_{k-1}y_kT_{k-1}=q(T_1 \cdots T_{k-1})(T_1  \cdots T_{k-2})y_ky_{k-1},
\]
\[
T_1\varphi^2=T_1(T_1 \cdots T_{k-1})y_k(T_1  \cdots T_{k-1})y_k=T_1(T_1  \cdots T_{k-1})(T_1  \cdots T_{k-2})y_kT_{k-1}y_k=
\]
\[
T_1(T_2\ldots T_{k-1})(T_1  \cdots T_{k-2})T_{k-1}y_kT_{k-1}y_k=q(T_1  \cdots T_{k-1})(T_1 \cdots T_{k-2})y_{k-1}y_{k}.
\]
Conversely, let us prove that \eqref{eq:phirel} imply \eqref{eq:Tyrel}. Define
\begin{equation}
\label{yidef}
y_i=q^{i-k}T_{i-1}^{-1}\cdots T_1^{-1}\varphi T_{k-1}\cdots T_{i}.
\end{equation}
Then, clearly, $T_{i}y_{i+1}T_{i}=qy_i.$ If $j>i$ then
\[
y_iT_j=q^{i-k}T_{i-1}^{-1}\cdots T_1^{-1}\varphi T_{k-1}\cdots T_{i}T_{j}=q^{i-k}T_{i-1}^{-1}\cdots T_1^{-1}\varphi T_{j-1}T_{k-1}\cdots T_{i}=
\]
\[
q^{i-k}T_{i-1}^{-1}\cdots T_1^{-1}T_{j}\varphi T_{k-1}\cdots T_{i}=T_{j}y_i.
\]
If $j<i-1$, the proof of $y_iT_j=T_{j}y_i$ is similar. Finally,   
\[
y_1y_k=\varphi T_{k-1}\cdots T_{1} T_{k-1}^{-1}\cdots T_{1}^{-1}\varphi=\varphi T_{k-2}^{-1}\cdots T_{1}^{-1}T_{k-1}\cdots T_{2}\varphi, 
\]
\[
y_ky_1=T_{k-1}^{-1}\cdots T_{1}^{-1}\varphi^2 T_{k-1}\cdots T_{1} =T_{k-1}^{-1}\cdots T_{1}^{-1}T_1\varphi^2 T_{k-2}\cdots T_{1} =
\]
\[
T_{k-1}^{-1}\cdots T_{2}\varphi^2 T_{k-2}\cdots T_{1}=\varphi T_{k-2}^{-1}\cdots T_{1}^{-1}T_{k-1}\cdots T_{2}\varphi.
\]
The proof of other commutation relations $y_iy_j=y_jy_i$ is similar.
\end{proof}

%21(-2)(-1)=(-1)121(-2)(-1)=(-1)212(-2)(-1)=(-1)2

%321(-3)(-2)(-1)=32(-3)1(-2)(-1)=(-2)232(-3)1(-2)(-1)=(-2)323(-3)1(-2)(-1)=(-2)321(-2)(-1)=(-2)3(-1)2=(-2)(-1)32

\begin{lemma}
\label{th:Iwahori}
The algebra $\BA_q$ is generated by $T_1,\ldots,T_{k-1},d_{+},d_{-}$ modulo relations \eqref{eq:Trel}, %%\todo{There was also (3.1.6) here which didn't make sense to me}
 all relations in \eqref{eq:Tdminusrel} and \eqref{eq:Tdplusrel} not involving $y_i$, and two additional relations:
\begin{equation}
\label{eq:phidplusdminus}
q \varphi d_- = d_- \varphi T_{k-1},\quad T_1 \varphi d_+ = q d_+ \varphi,
\end{equation}
where $\varphi=\frac{1}{q-1}[d_{+},d_{-}]$.
 All other relations follow from these. 
\end{lemma}

\begin{proof}
Let us check that $\varphi$ satisfies \eqref{eq:phirel} on $V_k$. Clearly, for $i\le k-2$ one has 
\[
(d_{+}d_{-}-d_{-}d_{+})T_i=d_{+}T_{i}d_{-}-d_{-}T_{i+1}d_{+}=T_{i+1}(d_{+}d_{-}-d_{-}d_{+}).
\]
Furthermore, 
\[
d_{+}d_{-}\varphi T_{k-1}=q d_{+}\varphi d_{-}=T_1\varphi d_{+}d_{-},
\]
and
\[
d_{-}d_{+}\varphi T_{k-1}=q^{-1} d_{-}T_1 \varphi d_{+} T_{k-1}=q^{-1} T_{1}d_{-}\varphi T_{k} d_{+}=T_1\varphi d_{-}d_{+},
\]
so 
\[
\varphi^2 T_{k-1}=\frac{1}{q-1}(d_{+}d_{-}-d_{-}d_{+})\varphi T_{k-1}=\frac{1}{q-1}T_{1}\varphi (d_{+}d_{-}-d_{-}d_{+})=T_1\varphi^2.
\]
Therefore by Lemma \ref{lem:Iwahori}  we can define $y_i$ and check the commutation relations \eqref{eq:Tyrel}.
Let us check the remaining relations:
\[
d_{-}y_i=d_{-}T_{i-1}^{-1}\cdots T_{1}^{-1}\varphi T_{k-1}\cdots T_{i}=T_{i-1}^{-1}\cdots T_{1}^{-1}d_{-}\varphi T_{k-1}\cdots T_{i}=
\]
\[
T_{i-1}^{-1}\cdots T_{1}^{-1}\varphi d_{-}T_{k-2}\cdots T_{i}=T_{i-1}^{-1}\cdots T_{1}^{-1}\varphi  T_{k-2}\cdots T_{i}d_{-}=y_id_{-}.
\]
The last identity $d_{+}y_i=T_1 \cdots T_i y_i T_i^{-1} \cdots T_1^{-1} d_+$ is also straightforward, see \cite[Lemma 5.4]{carlsson2015proof}.
\end{proof}

 \subsection{$\BA_{q,t}$}
The ``double algebra'' $\BA_{q,t}$ depends on two parameters $q,t\in\BQ(q,t)$ and is obtained from two copies of $\BA_{q}$ by imposing more relations: 
\begin{definition}
$\BA_{q,t}$ is the $\BQ(q,t)$-linear algebra generated by a collection of orthogonal idempotents labelled by $\BZ_{\geq 0}$ and elements:
\[
d_+,d_+^*:k\to k+1,\; d_-:k\to k-1,\; T_i:k\to k \quad(1\leq i<k),\; y_i,z_i:k\to k \quad(1\leq i\leq k)
\]
subject to the
\begin{itemize}
\item relations of $\BA_q$ for $d_-, d_+, T_i, y_i$,
\item relations of $\BA_{q^{-1}}$ for $d_-, d_+^*, T_i^{-1}, z_i$,
\end{itemize}
and 
\begin{equation}\label{eq:yzrel}
d_+ z_i = z_{i+1} d_+, \quad d_+^* y_i = y_{i+1} d_+^* \quad(1\leq i \leq k),\quad z_1 d_+ = - t q^{k+1} y_1 d_+^*.
\end{equation}
\end{definition}

\begin{remark}
One is tempted to say that the generators $T_i$, $y_i$ and $z_i$ form some sort of double affine Hecke algebra as in Remark \ref{rem:heckecat}, but this is not the case. The problem stems from the fact that double affine Hecke algebras of \cite{cherednik2005double} do not embed into one another in the way that the affine Hecke algebras do. There is a way, %%\todo{Details? Included in CM shuffle paper by now}
however, to relate $\BA_{q,t}$ to double affine Hecke algebras by making sense of limits of the form $\lim_{n\to \infty} e_n \mathrm{DAHA}_{n+k} e_n$, where $e_n\in \mathrm{DAHA}_{n+k}$ is the partial symmetrization operator on indices $k+1$, $k+2$, \ldots, $k+n$.
\end{remark} 

%%\subsection{$\BB_{q,t}$} 
In what follows we will need a certain subalgebra of $\BA_{q,t}$ which, nevertheless, contains an isomorphic copy of $\BA_{q,t}$.%%\todo{check powers of $q$ and other factors everywhere}.

%\begin{lemma}
%\label{lem:z rels}
%The following relations between the operators $d_+,d_-$ and $z_i$ are satisfied in $\BA_{q,t}$:
%\begin{equation}
%\label{eq: z eqns}
%z_i d_- = d_- z_i,\quad d_+ z_i = z_{i+1} d_+,
%\end{equation}
%\[
%z_1 (q d_+ d_- - d_- d_+) = qt (d_+ d_- - d_- d_+) z_k.
%\]
%\end{lemma}

%\begin{proof}
%The first two relations follow from the definition of $\BA_{q,t}$. For the last one, we have:
%\[
%z_1 (q d_+ d_- - d_- d_+)=qz_1d_+d_{-}-z_1d_{-}d_{+}=q(z_1d_+)d_{-}-d_{-}(z_1d_{+}).
%\]
%We can replace $z_1d_+$ by a multiple of $y_1 d_+^*$ and obtain:
%\[
%q(-tq^{k})y_1d_+^*d_{-}-d_{-}(-tq^{k+1})y_1d_+^*=-tq^{k+1}y_1[d_+^*,d_{-}].
%\]
%Since  $d_-, d_+^*, T_i^{-1}, z_i$ satisfy the relations for $\BA_{q^{-1}}$, by \eqref{eq:commrel} we get:
%\[
%[d_+^*, d_-]= (q^{-1}-1) T_1^{-1} \cdots T_{k-1}^{-1} z_k,
%\]
%so
%\[
%-tq^{k+1}y_1[d_+^*,d_{-}]=-tq^{k+1}(q^{-1}-1)y_1 T_1^{-1} \cdots T_{k-1}^{-1} z_k=
%\]
%\[
%tq^{k}(q-1)y_1 T_1^{-1} \cdots T_{k-1}^{-1} z_k =
%qt(q-1)T_1\cdots T_{k-1}y_kz_k=qt[d_{+},d_{-}]z_{k}.
%\]
%\end{proof}

\begin{definition}
The algebra $\BB_{q,t}$ is generated by a collection of orthogonal idempotents labelled by $\BZ_{\geq 0}$, generators  $d_{+},d_{-},T_i$ and $z_i$ modulo relations:
\[
(T_i-1)(T_i+q)=0, \quad T_i T_{i+1} T_i = T_{i+1} T_i T_{i+1}, \quad T_i T_j = T_j T_i \quad (|i-j|>1),
\]
\[
T^{-1}_i z_{i+1} T^{-1}_i = q^{-1} z_{i} \; (1\leq i\leq k-1),
\]
\[
z_i T_j = T_j z_i  \; (i\notin \{j, j+1\}),\;
z_i z_j = z_j z_i \; (1\leq i,j \leq k),
\]
\[
d_-^2 T_{k-1} = d_-^2,\;  d_- T_i = T_i d_-\; (1\leq i \leq k-2),\;
\]
\[
T_1 d_+^2 = d_+^2,\; d_+ T_i = T_{i+1} d_+ \; (1\leq i \leq k-1),
\]
%%\[
%%\varphi T_i=T_{i+1}\varphi\ (i\le k-2),\ \varphi^2 T_{k-1}=T_1\varphi^2.
%%\]
%%\todo{do we really need these two above?}
\[
q \varphi d_- = d_- \varphi T_{k-1},\quad T_1 \varphi d_+ = q d_+ \varphi,
\]
\[
z_i d_- = d_- z_i,\quad d_+ z_i = z_{i+1} d_+,
\]
\[
z_1 (q d_+ d_- - d_- d_+) = qt (d_+ d_- - d_- d_+) z_k.
\]
\
\end{definition}

\begin{remark}
By \eqref{eq:commrel}, one can define the elements $y_i\in \BB_{q,t}$ and prove that $y_i,T_i,d_+$ and $d_-$ generate a copy of $\BA_{q}$.
\end{remark}

\begin{proposition}
There is a homomorphism  $\alpha:\BB_{q,t}\to \BA_{q,t}$ which sends $d_{+},d_{-},T_i$ and $z_i$ to the 
corresponding generators of $\BA_{q,t}$.
\end{proposition}

\begin{proof}
Let us check that the  last defining relation for $\BB_{q,t}$ holds in $\BA_{q,t}$:
\[
z_1 (q d_+ d_- - d_- d_+)=qz_1d_+d_{-}-z_1d_{-}d_{+}=q(z_1d_+)d_{-}-d_{-}(z_1d_{+}).
\]
We can replace $z_1d_+$ by a multiple of $y_1 d_+^*$ and obtain:
\[
q(-tq^{k})y_1d_+^*d_{-}-d_{-}(-tq^{k+1})y_1d_+^*=-tq^{k+1}y_1[d_+^*,d_{-}].
\]
Since  $d_-, d_+^*, T_i^{-1}, z_i$ satisfy the relations for $\BA_{q^{-1}}$, by \eqref{eq:commrel} we get:
\[
[d_+^*, d_-]= (q^{-1}-1) T_1^{-1} \cdots T_{k-1}^{-1} z_k,
\]
so
\[
-tq^{k+1}y_1[d_+^*,d_{-}]=-tq^{k+1}(q^{-1}-1)y_1 T_1^{-1} \cdots T_{k-1}^{-1} z_k=
\]
\[
tq^{k}(q-1)y_1 T_1^{-1} \cdots T_{k-1}^{-1} z_k =
qt(q-1)T_1\cdots T_{k-1}y_kz_k=qt[d_{+},d_{-}]z_{k}.
\]
It follows from the definition, Theorem \ref{th:Iwahori} that %and Lemma \ref{lem:z rels} that
 all other defining relations of $\BB_{q,t}$ are satisfied
in $\BA_{q,t}$.
\end{proof}

\begin{theorem}
There is an algebra homomorphism $\beta:\BA_{q,t}\to \BB_{q,t}$ such that 
\[
\beta(T_i)=T_i,\;\beta(d_-)=d_-,\;\beta(d_+)=d_+,\;\beta(d_+^*)=q^{-k}z_1d_{+}
\]
and $\beta(z_1)=-qt y_1z_1$. There is a chain of homomorphisms:
\[
\BA_{q,t}\xrightarrow{\beta}\BB_{q,t}\xrightarrow{\alpha}\BA_{q,t}.
\]
\end{theorem}

%\begin{remark}
%The composition $\alpha\circ\beta:\BA_{q,t}\to \BA_{q,t}$ is a part of $SL(2,\BZ)$ action on $\BA_{q,t}$.
%It was denoted by ?? in \cite{mellit2016toric}.
%\end{remark}

%%\todo{check all constants!!}

\begin{proof}
It is clear that all defining relations of $\BA_{q}$ are satisfied for $T_i,d_-,d_+$ and hence for $y_i$. We proceed to check the relations of $\BA_{q^{-1}}$ for $T_i^{-1},d_-,\beta(d_+^*), z_i$ in $\BB_{q,t}$. In order to apply Lemma \ref{lem:Iwahori} we will need the following computation:
\[
(q^{-1}-1) \varphi^* = [\beta(d_+^*),d_-]
=q^{1-k}z_1d_{+}d_{-}-q^{-k}z_1 d_{-}d_{+}=
q^{-k}z_1(qd_{+}d_{-}-d_{-}d_{+})
\]
\[
=t q^{1-k}(d_+ d_- - d_- d_+) z_k= t q^{1-k}(q-1) \varphi z_k.
\]
Thus we have
\[
\varphi^* = -t q^{2-k} \varphi z_k,
\]
so that we can check \eqref{eq:phidplusdminus}:
\[
q^{-1} \varphi^* d_- = -t q^{2-k} \varphi z_{k-1} d_-= -t q^{1-k}  d_- \varphi T_{k-1} z_{k-1} = d_-\varphi^*  T_{k-1}^{-1},
\]
\[
T_1^{-1} \varphi^* \beta(d_+^*) = - t q^{1-k} T_1^{-1} \varphi z_{k+1} q^{-k} z_1 d_+ = -t q^{1-2k} T_1^{-1} \varphi z_1 d_+ z_k
\]
\[
= -t q^{1-2k} T_1^{-1} z_2 \varphi d_+ z_k = -t q^{-2k} z_1 T_1 \varphi d_+ z_k = -t q^{1-2k} z_1 d_+ \varphi z_k = q^{-1} \beta(d_+^*) \varphi^*,
\]
where we have used the following identity between elements $k\to k$ for $k\geq 2$:
\begin{equation}\label{eq:phi z1}
\varphi z_1 = \frac{1}{q-1} (d_+ d_- - d_- d_+) z_1 = \frac{1}{q-1} z_2 (d_+ d_- - d_- d_+) = z_2\varphi.
\end{equation}
Among prerequisites for Lemma \ref{lem:Iwahori} it remains to check the identities between $\beta(d_+^*)$ and $T_i$. We have
\[
\beta(d_+^*) T_i = q^{-k} z_1 d_+ T_i = q^{-k} z_1 T_{i+1} d_+ = q^{-k} T_{i+1} d_+ = T_{i+1} \beta(d_+^*),
\]
\[
\beta(d_+^*)^2 =  q^{-2k-1} z_1 d_+ z_1 d_+ = q^{-2k-1} z_1 z_2 d_+^2,
\]
hence
\[
T_1 \beta(d_+^*)^2 =  q^{-2k-1} z_1 z_2 T_1 d_+^2 = q^{-2k-1} z_1 z_2 d_+^2 = \beta(d_+^*)^2.
\]
Thus we can apply Lemma \ref{lem:Iwahori} and deduce that the relations of 
$\BA_{q^{-1}}$ for $T_i^{-1},d_-,\beta(d_+^*), z_i$ are satisfied. 

It remains to check relations \eqref{eq:yzrel} for $d_+, y_i, \beta(d_+^*), \beta(z_i)$. We have
\[
\beta(z_k) = T_{k-1} \cdots T_1 \varphi^* = -t q^{2-k} T_{k-1} \cdots T_1 \varphi z_k.
\]
Therefore
\[
\beta(z_i)= -t q^{2-k} T_{i-1}\cdots T_{1} \varphi T_{k-1} \cdots T_i z_i = -qt T_{i-1}\cdots T_{1} y_1 T_1^{-1} \cdots T_{i-1}^{-1} z_i.
\]
Thus we have
\[
d_+ \beta(z_i) = -qt d_+ T_{i-1}\cdots T_{1} y_1 T_1^{-1} \cdots T_{i-1}^{-1} z_i = -qt T_{i}\cdots T_{1} y_1 T_1^{-1} \cdots T_{i}^{-1} z_{i+1}
\]
\[
 = \beta(z_{i+1}) d_+. 
\]
Using Lemma \ref{lem:Iwahori} and \eqref{eq:phi z1} we obtain
\[
\beta(d_+^*) \varphi = q^{-k} z_1 d_+ \varphi = q^{-1-k} z_1 T_1 \varphi d_+ = q^{-k} T_1^{-1} z_2 \varphi d_+ = q^{-k} T_1^{-1} \varphi z_1 d_+
\]
\[
 = T_1^{-1}\varphi \beta(d_+^*),
\]
which implies
\[
\beta(d_+^*) y_i = \beta(d_+^*) T_{i-1}^{-1}\cdots T_{1}^{-1}\varphi T_{k-1}\cdots T_{i} = T_i^{-1}\cdots T_1^{-1}\varphi T_{k}\cdots T_{i+1} \beta(d_+^*) 
\]
\[
= y_{i+1} \beta(d_+^*).
\]
Finally, we have
\[
\beta(z_1) d_+ = -q t y_1 z_1 d_+ = -t q^{k+1} y_1 \beta(d_+^*).
\]
Thus we finished verifying \eqref{eq:yzrel}.
\end{proof}

\subsection{Gradings}

The algebras $\BA_{q,t}$ and $\BB_{q,t}$ are triply graded. The grading of $d_+$ is $(1,0,0)$, the grading of $d_-$ is $(0,1,0)$,
and the grading of $T_i$ is $(0,0,0)$. The commutation relations imply that $y_i$ have grading $(1,1,0)$. Next, we require that 
$d_+^*$ has grading $(0,0,1)$ and $z_i$ have grading $(0,1,1)$. It is easy to check that all relations are tri-homogeneous with 
respect to these gradings. In particular, the degrees of $z_1d_{+}$ and $y_1d_+^*$ are both equal to $(1,1,1)$. 

In what follows we will use two specializations of this triple grading. The first projection $(a,b,c)\mapsto a-b+c$ assigns to $d_+$
and $d_{+}^*$ degree 1, $d_-$ has degree $(-1)$ and $y_i,z_i,T_i$ all have degree 0. This is just the standard grading which equals $k$ in the idempotent $e_k$.

The more interesting projection $(a,b,c)\to a+b+c$ assigns to $d_+,d_-,d_+^*$ degree 1, and to $y_i,z_i$ degree 2. 

\subsection{Polynomial representation}

Denote by $\Lambda$ the ring of symmetric functions in $x_1,x_2,\ldots$. Following \cite{carlsson2015proof} we introduce spaces 
\[
V_k = \Lambda \otimes \BC(q,t)[y_1,\cdots,y_k],\quad V_\bullet = \bigoplus_{k\geq 0} V_k.
\]
One of the results of \cite{carlsson2015proof} is the following:
\begin{proposition}
\label{halfalgprop}
There is an action of $\BA_{q,t}$ on $V_\bullet$ in which
\[
T_i F=\frac{(q-1)y_{i+1}F+(y_{i+1}-qy_i)s_iF}{y_{i+1}-y_{i}},\quad y_i F=y_i\cdot F,
\]
\[
d_- F = -\Res_{y_k} F[X-(q-1) y_k] \pExp[-y_k^{-1} X] d y_k \quad(F\in V_k),
\]
\[
d_+ F = T_1 T_2 \cdots T_k (F[X+(q-1) y_{k+1}]).
\]
\[
d_{+,CM}^* F=\gamma F[X+(q-1)y_{k+1}],
\]
where $\gamma(y_i)=y_{i+1}$ and $\gamma(y_{k+1})=ty_1$.
Furthermore, we have a unique isomorphism 
\[
V_\bullet=\BA_q \Id_0,
\]
of left $\BA_q$-modules in which $1\in V_0$ maps to $\Id_0$.
\end{proposition}

Consider the space 
\[
W_\bullet :=\bigoplus W_k,\quad W_k=(y_1\cdots y_k)^{-1}V_k
\]
Clearly, $V_k\subset W_k$.

\begin{theorem}
The following statements hold:

1) The operators $T_i$, $d_{+}$, $d_{-}$ and $d_{+}^{*}=-(q t y_1)^{-1}d_{+,CM}^{*}$ can be naturally extended to the space $W_\bullet$ and  
define a representation of $\BA_{q,t}$.

2) In this representation, $\alpha(\BB_{q,t})$ preserves the subspace $V_\bullet\subset W_\bullet$, and hence defines a representation 
of $\BB_{q,t}$ in $V_\bullet$.

3) The composition $\alpha\beta(\BA_{q,t})$ also preserves $V_\bullet$, and hence defines a representation 
of $\BA_{q,t}$ in $V_\bullet$. This representation agrees with the one in Proposition \ref{halfalgprop}.
\end{theorem}

We illustrate all these representations in the following commutative diagram:
\begin{center}
\begin{tikzcd}
\BA_{q,t}\arrow[r,"\beta"]  \arrow[d,"d_{+,CM}^*"] & \BB_{q,t}\arrow[r,"\alpha"] \arrow[d] & \BA_{q,t} \arrow[d,"d_+^*"]\\
\End(V_\bullet) & \arrow[l] \End_{V_\bullet}(W_\bullet) \arrow[r,hook] & \End(W_\bullet)\\
\end{tikzcd}
\end{center}
Here $\End_{V_\bullet}(W_\bullet)$ denotes the set of endomorphisms of $W_\bullet$ preserving $V_\bullet$. 

\begin{proof}
Let us prove that $T_i$, $d_{+}$, $d_{-}$ and $d_{+}^{*}=-(qt y_1)^{-1}d_{+,CM}^{*}$ are well-defined on $W_\bullet$. 
If $F\in V_\bullet$, then 
\[
T_i(F/(y_1\cdots y_k))=(T_iF)/(y_1\cdots y_k)\in W_\bullet,
\]
\[
d_+(F/(y_1\cdots y_k))=(T_1 T_2 \cdots T_k y_{k+1}F[X+(q-1) y_{k+1}])/(y_1\cdots y_{k+1})\in W_\bullet,
\]
\[
d_-(F/(y_1\cdots y_k))=-(y_1\cdots y_{k-1})^{-1}\Res_{y_k} F[X-(q-1) y_k]y_k^{-1}\pExp[-y_k^{-1} X] d y_k\in W_\bullet,
\]
\[
-qt d_{+}^{*}(F/(y_1\cdots y_k))=-y_1^{-1}\gamma \left(F[X+(q-1)y_{k+1}]/(y_1\cdots y_k)\right)=
\]
\[
(y_1\cdots y_{k+1})^{-1}\gamma (F[X+(q-1)y_{k+1}]\in W_\bullet. 
\]
The verification of commutation is identical to \cite{carlsson2015proof} and we leave it to the reader.%% Maybe a reference to Anton's paper

To prove that $\alpha(\BB_{q,t})$ preserves $V_\bullet$, it is sufficient to prove that the commutator $[d_{+}^{*},d_{-}]$ preserves $V_\bullet$ (then $z_i$ preserve $V_\bullet$,
and $T_i,d_+,d_-$ preserve $V_\bullet$ by definition). For $F\in V_k$ we have:
\[
-qt d_{+}^{*}d_{-}F=-y_1^{-1}F[X+(1-q)ty_1-(q-1)u,y_2,\ldots,y_k,u]
\]
\[
\times \pExp[-u^{-1}X-u^{-1}(q-1)ty_1]|_{u^{-1}},
\]
\[
-qt d_{-}d_{+}^*=-y_1^{-1}F[X+(1-q)ty_1-(q-1)u,y_2,\ldots,y_k,u]\pExp[-u^{-1}X]|_{u^{-1}}.
\]
Now 
\[
\frac{1-u^{-1}qty_1}{1-u^{-1}ty_1}-1=(1-q)\frac{u^{-1}ty_1}{1-u^{-1}ty_1},
\]
so
\[
[d_{+}^*,d_{-}]F=(1-q^{-1})F[X+(1-q)ty_1-(q-1)u,y_2,\ldots,y_k,u]\pExp[u^{-1}ty_1-u^{-1}X]|_{u^0}.
\]
Finally, $\alpha\beta(d_{+}^{*})=-qty_1d_{+}^{*}=d_{+,CM}^{*}$.
\end{proof}

This result is very useful in the proof of our main theorem. Namely, we will define a geometric representation of $\BB_{q,t}$
and identify it with the space $V_\bullet$. Then, using the homomorphism $\beta$, we will define a representation of $\BA_{q,t}$ which, by the above, is isomorphic to the representation from Proposition \ref{halfalgprop}.

Finally, a key observation from \cite{carlsson2015proof} is that there is a symmetry
in the relations of $\mathbb{A}_{q,t}$ which
is antilinear with respect to the conjugation $(q,t)\mapsto (q^{-1},t^{-1})$,
and is given on generators by
\begin{equation}
  \label{invsymmetry}
  d_-\leftrightarrow d_-,\quad T_i \leftrightarrow T_i^{-1},
  \quad y_i \leftrightarrow z_i,\quad
  d_+\leftrightarrow d_+^*
  \end{equation}
Furthermore, this symmetry preserves the kernel of the map
$\mathbb{A}_{q,t}\rightarrow \End(V_\bullet)$,
and so determines a map
\begin{equation}
  \label{NVdef}
  \mathcal{N} : V_\bullet \rightarrow V_\bullet
  \end{equation}
which is antilinear, and satisfies
$\mathcal{N}^2=1$.

\section{The spaces} 

 \subsection{Parabolic flag Hilbert schemes}
 
\begin{definition}
The parabolic flag Hilbert scheme $\PFH_{n,n-k}$ of points on $\BC^2$ is the moduli space of flags
\[
I_n\subset I_{n-1}\subset \ldots \subset I_{n-k} 
\]
where $I_{n-i}$ is the ideal in $\BC[x,y]$ of codimension $(n-i)$ and $yI_{n-k}\subset I_{n}$.
\end{definition}

\begin{definition}
\label{def: ADHM}
The parabolic flag Hilbert scheme $\PFH_{n,n-k}$ of points on $\BC^2$ is the space of triples $(X,Y,v)/G$
where $v\in BC^n$, $X$ and $Y$ are $(n-k,k)$ block lower-triangular  matrices such that
$k\times k$ block is lower-triangular in $X$ and vanishes in $Y$: 
\begin{equation}
\label{X and Y}
X= \left(
\begin{tabular}{c|ccr}
\\
  * &  & 0 & \\
  \vdots  & \vdots
  & \qquad \\
 \hline
  \qquad \qquad \qquad & * & 0  & \ldots  0 \\
   \qquad \qquad \qquad & *  & * &  \ldots  0 \\
      \vdots & \vdots & $\vdots$ & \vdots\, \\
   \qquad \qquad \qquad & * & $\cdots$ & * \\
   
\end{tabular}
\right),\;
Y= \left(
\begin{tabular}{c|ccc}
\\
  * &    0   \\
 \qquad \qquad \qquad & \qquad \\
 \hline
   \qquad \qquad \qquad &   0 \ldots  0 \\
      \vdots & \vdots   \\
   \qquad \qquad \qquad & 0 \ldots  0 \\
   
\end{tabular}
\right).
\end{equation}
We require that $[X,Y]=0$
and the stability condition $\BC\langle X,Y \rangle v= \BC^{n}$ holds. The group $G$ consists of $(n-k,k)$ invertible block lower-triangular matrices with lower-triangular $k\times k$ block, and acts
by $g.(X,Y,v)=(gXg^{-1},gYg^{-1},gv).$
\end{definition}

\begin{proposition}
Two definitions of $\PFH_{n,n-k}$ are equivalent.
\end{proposition}
\begin{proof}
The proof is standard but we include it here for completeness.
Given a flag of ideals $\{I_n\subset I_{n-1}\subset \ldots \subset I_{n-k}\subset \BC[x,y]\}$, consider the sequence of vector spaces 
$W_s=\BC[x,y]/I_s$.  The multiplication by $x$ and $y$ induces an action of two commuting operators $X$ and $Y$ on each $W_s$. 
There is a sequence of surjective maps $W_{n}\twoheadrightarrow W_{n-1}\twoheadrightarrow \ldots \twoheadrightarrow W_{n-k}$ which commute with the action of $X$ and $Y$. Since $yI_{n-k}\subset I_{n}$, the operator $Y$ annihilates 
\[
\Ker(W_{n}\twoheadrightarrow W_{n-k})=I_{n-k}/I_{n}.
\]
If one chooses a basis in all $W_s$ compatible with the projections, then the operators $X$ and $Y$ in this basis would have the form \eqref{X and Y}. The vector $v$ corresponds to the projection of $1\in \BC[x,y]$, and the matrix $g$ corresponds to the change of basis. 

Conversely, given a triple $X,Y,v$, let $W_s$ be the vector space spanned by the first $s$ coordinate vectors, and let $X_s,Y_s,v_s$ denote the restrictions of $X, Y$ and $v$ to $W_s$. Let $I_s=\{f\in \BC[x,y]\ :\ f(X_s,Y_s)(v_s)=0\}.$ Clearly, $I_s$ is an ideal, $I_{s+1}\subset I_s$
and $yI_{n-k}\subset I_n$. 
\end{proof}

\begin{example}
If $k=0$ then clearly $\PFH_{n,n-k}=\Hilb^n(\BC^2)$. If $k=n$ then $\PFH_{n,n-k}=\BC^n$. Indeed,  for $k=n$ the matrix $Y$ vanishes, and the stability condition implies that $X$ is determined up to conjugation by its eigenvalues (that is, all generalized eigenvectors with the same eigenvalue belong to a single Jordan block). Therefore the natural projection 
\[
\PFH_{n,0}\to \BC^n,\ (X,Y,v)\mapsto (x_{11},\ldots,x_{nn})
\]
is an isomorphism.
\end{example}

These examples indicate that $\PFH_{n,n-k}$ behaves better than the full flag Hilbert scheme which is very singular \cite{GNR}.
This is indeed true in general.

\begin{theorem}
\label{th:smooth}
The space $\PFH_{n,n-k}$ is a smooth manifold of dimension $2n-k$ for all $n$ and $k$.
\end{theorem}
 
In the proof of this theorem we will use a version of the geometric construction of  Biswas and Okounkov \cite{biswas1997parabolic}(see also \cite[Section 3.4]{feigin2011yangians}, \cite[Section 4.3]{NKuzn} and references therein). Consider the map
\[
\sigma:\BC^2\to \BC^2,\ \sigma(x,y)=(x,y^{k+1}).
\]
Also, consider an action of the group $\Gamma=\BZ/(k+1)\BZ$ on $\BC^2$ given by $(x,y)\mapsto (x,\zeta y)$,
where $\zeta$ is a primitive $(k+1)$st root of unity. %Clearly,
%$$
%\sigma(\zeta(x,y))=\sigma(x,y).
%$$
Given a sequence of ideals $I_n,\ldots,I_{n-k}$, we can consider the space
\[
J(I_{n},\ldots,I_{n-k})=\sigma^{*}I_{n}+y\sigma^{*}I_{n-1}+\ldots+y^{k}\sigma^{*}I_{n-k}\subset \BC[x,y]
\]  
\begin{lemma}\label{lem:space is an ideal}
The space $J(I_{n},\ldots,I_{n-k})$ is an ideal in $\BC[x,y]$ if and only if $yI_{n-k}\subset I_{n}\subset I_{n-1}\subset\cdots\subset I_{n-k}$.
\end{lemma}
\begin{proof}
Clearly, multiplication by $x$ preserves the space $J(I_{n},\ldots,I_{n-k})$, so it is an ideal if and only if it is preserved by the multiplication by $y$. For $0\le j< k$ one has
\[
y\cdot y^{j}\sigma^{*}I_{n-j}=y^{j+1}\sigma^{*}I_{n-j}%\subset y^{j+1}\sigma^{*}I_{n-j-1},
\]
which is contained in $y^{j+1}\sigma^{*}I_{n-j-1}$ if and only if $I_{n-j}\subset I_{n-j-1}$. Furthermore, 
\[
y\cdot y^{k}\sigma^*I_{n-k}=y^{k+1}\sigma^*I_{n-k}=\sigma^*(yI_{n-k}),
\] 
which is contained in $\sigma^{*}I_{n}$ if and only if $yI_{n-k}$ is contained in $I_{n}$.
\end{proof}
\begin{lemma}\label{lem:an ideal is invariant}
An ideal $J\subset \BC[x,y]$ is invariant under the action of $\Gamma$ if and only if $J=J(I_{n},\ldots,I_{n-k})$
for some ideals $I_n\subset\cdots\subset I_{n-k}$ with $y I_{n-k}\subset I_n$. In this case the ideals $I_{n-j}$ are uniquely determined by $J$.
\end{lemma}
\begin{proof}
Clearly,  $\sigma^*\BC[x,y]=\BC[x,y^{k+1}]\subset \BC[x,y]$ is invariant under the action of $\Gamma$, so
$J(I_{n},\ldots,I_{n-k})$ is also invariant. Conversely, let $J$ be a $\Gamma$-invariant ideal in $\BC[x,y]$, we can decompose it according to the action of $\Gamma$:
\[
J=\oplus_{s=0}^{k}J^{(s)},\ \zeta(f)=\zeta^{s}f\ \text{for}\ f\in J^{(s)}.
\]
Since $y^{k+1}J^{(s)}\subset J^{(s)}$, we can write $J^{(s)}=y^{s}\sigma^*(I_{n-s})$ for some ideal $I_{n-s}$.
By Lemma \ref{lem:space is an ideal}, $I_{n-s}\subset I_{n-s-1}$ and $yI_{n-k}\subset I_{n}$.
\end{proof}

\begin{proof}[Proof of Theorem \ref{th:smooth}]
By Lemma \ref{lem:an ideal is invariant}, the space $\PFH_{n,n-k}$ can be identified with a subset of the fixed point set of the action of a finite group $\Gamma$ on the Hilbert scheme $\Hilb^n(\BC^2)$. The codimensions of $I_{n-s}$ are locally constant functions on the fixed point set. Therefore $\PFH_{n,n-k}$ can be identified with a union of several connected components of the fixed point set. Since $\Hilb^n(\BC^2)$ is smooth, the fixed point set is also smooth.%% \todo{Is $\PFH_{n,n+k}$ connected?EG: I doubt that}
\end{proof}

 \subsection{Torus action}

The group $\BT=\BC^*\times \BC^*$ acts on $\BC^2$ by scaling the coordinates: $(x,y)\to (q^{-1}x,t^{-1}y)$. This action can be lifted to the action on the Hilbert schemes $\Hilb^n$ and the spaces $\PFH_{n,n-k}$. The fixed points of this action on $\Hilb^n$ correspond to monomial ideals $I_{\lambda}$ and are labeled by Young diagrams $\lambda$ with $|\lambda|=n$.  It is convenient to encode a single cell $\square$ by its monomial $\chi(\square) = q^c t^r$, where $c$ resp. $r$ is the column resp. row index of $\square$. It is well known (e.g. Lemma 5.4.5 in \cite{haiman2002combinatorics}, see also \cite{nakajima1999lectures}) that the equivariant character of the cotangent space at $I_{\lambda}$ is given by 
\begin{equation}
%\label{tangent hilb}
\label{eq: arms legs identity}
\ch \Omega_{I_{\lambda}}\Hilb^n=\sum_{\sq\in \lambda}(q^{a(\sq)+1}t^{-l(\sq)}+q^{-a(\sq)}t^{l(\sq)+1})= qt B_\mu + B_\mu^*-(q-1)(t-1)B_\mu B_\mu^*,
\end{equation}
where $a(\sq)$ and $l(\sq)$ denote the lengths of the arm and the leg of $\sq$ in $\lambda$, $B_\mu = \sum_{\square\in\mu} \chi(\square)$ and $*$ in $B_\mu^*$ denotes the substitution $q\to q^{-1}$, $t\to t^{-1}$.

The fixed points of $\PFH_{n,n-k}$ are labeled by sequences of monomial ideals $I_{n}\subset \ldots \subset I_{n-k}$ corresponding to Young diagrams $\lambda^{(n)}\supset \ldots\supset \lambda^{(n-k)}$. The condition $yI_{n-k}\subset I_{n}$ can be translated to $\lambda^{(i)}$
as follows: $\lambda^{(n)}\setminus \lambda^{(n-k)}$ is a (possibly disconnected) {\bf horizontal strip}, that is, it contains at most one box in each column. Another useful reformulation of this condition is 
\begin{equation}
\label{cyclic}
\lambda^{(n-k)}_i\ge \lambda^{(n)}_{i+1},\ \text{where}\ \lambda^{(n-j)}=(\lambda^{(n-j)}_1 \ge \lambda^{(n-j)}_2\ge \ldots). 
\end{equation}
Note that the difference $\lambda^{(n-j)}\setminus \lambda^{(n-j-1)}$ consists of a single box. Instead of keeping track of the sequence of partitions we prefer to remember only the first one, which we denote by $\lambda=\lambda^{(n)}$, and the successive differences $\square_j = \lambda^{(n-j+1)}\setminus \lambda^{(n-j)}$ ($j=1,\ldots,k$). When drawing a picture we will display $\lambda$ as a Young diagram, together with labeling of some of its cells by numbers from $1$ to $k$ where we put $j$ in $\square_j$. 
Alternatively, we will form a vector $w=(w_1,\ldots,w_k)$ where $w_j=\chi(\square_j)$.
%So we will encode the cells $\square_j$ by their monomials $w_j=\chi(\square_j)$ forming a vector $w=(w_1,\ldots,w_k)$. 
A fixed point in $\PFH_{n,n-k}$ will be denoted by $I_{\lambda,w}$ when we specify a pair of a partition $\lambda$ and a vector $w$, or by $I_{\lambda^{(\bullet)}}$ when we specify a decreasing sequence of partitions $\lambda^{(\bullet)}$.

Another way of encoding sequences of partitions $\lambda^{(n-j)}$ comes from the proof of Theorem \ref{th:smooth}. If all $I_{n-j}$ are monomial ideals, so is $J(I_{n},\ldots,I_{n-k})$. The corresponding Young diagram $\mu$ has rows:
\[
\mu=(\lambda^{(n)}_1,\ldots,\lambda^{(n-k)}_1,\lambda^{(n)}_2,\ldots,\lambda^{(n-k)}_2,\lambda^{(n)}_3,\ldots),
\]
which decrease by \eqref{cyclic}.
%We have the identity
%\begin{equation}\label{eq: arms legs identity}
%\ch \Omega_{I_{\mu}}\Hilb=\sum_{\sq\in \mu}(q^{a(\sq)+1}t^{-l(\sq)}+q^{-a(\sq)}t^{l(\sq)+1}) = qt B_\mu + B_\mu^*-(q-1)(t-1)B_\mu B_%\mu^*,
%\end{equation}
% We are going to apply \eqref{eq: arms legs identity} to the partition $\mu$ which arises from a sequence of partitions $\lambda^{(n)}\supset\cdots\supset \lambda^{(n+k)}$ in the proof of Theorem \ref{th:smooth}. 
Note that
\[
B_\mu = B_{\lambda^{(n)}}(q, t^{k+1}) + t B_{\lambda^{(n-1)}}(q, t^{k+1}) +\cdots+t^k B_{\lambda^{(n-k)}}(q, t^{k+1}).
\]
To calculate the character of $\Omega_{\lambda_\bullet} \PFH_{n,n+k}$ we need to extract the terms in $\ch \Omega_{I_{\mu}}\Hilb$ whose $t$-degree is divisible by $k+1$, and then replace each term $q^{a} t^{b(k+1)}$ by $q^a t^b$. Performing this with \eqref{eq: arms legs identity} we obtain:
\[
qt B_{\lambda^{(n-k)}} + B_{\lambda^{(n)}}^* + (q-1) \left(\sum_{i=0}^{k} B_{\lambda^{(n-i)}} B_{\lambda^{(n-i)}}^* - t B_{\lambda^{(n-k)}} B_{\lambda^{(n)}}^* - \sum_{i=1}^{k} B_{\lambda^{(n-i+1)}} B_{\lambda^{(n-i)}}^*\right),
\]
which can be rewritten as
\[
qt B_{\lambda^{(n-k)}} + B_{\lambda^{(n)}}^* + (q-1)\left((B_{\lambda^{(n)}}-t B_{\lambda^{(n-k)}})B_{\lambda^{(n)}}^* - \sum_{i=1}^{k} w_{i} B_{\lambda^{(n-i)}}^*\right),
\]
so we obtain
\begin{equation}\label{eq:tangent char}
\ch \Omega_{\lambda^{(\bullet)}} \PFH_{n,n-k} = qt B_{\lambda^{(n-k)}} + B_{\lambda^{(n)}}^* - (t-1)(q-1) B_{\lambda^{(n-k)}} B_{\lambda^{(n)}}^* + (q-1) \sum_{k\geq i\geq j \geq 1} w_i w_j^{-1}.
\end{equation}
% %I comment this out because it is $0/0$:
%Sometimes we will need a different presentation of the tangent space or, equivalently, its exterior algebra. 
%Let $\omega(x)=\frac{(1-x)(1-qtx)}{(1-tx)(1-qx)}$.
%
%%%\todo{Check order of w_i}
%\begin{corollary}%\todo{deprecated}
%\label{prop:product}
%Let   $w_i$ denote the $(q,t)$ character of the box labeled by $i$.  Then 
%\begin{multline}
%\ch \wedge T_{\lambda_{\bullet}}\PFH_{n,n-k}=\frac{1}{(1-q)^{k}}\prod_{\sq\in \lambda}\frac{1}{(1-\sq)(1-qt\sq)}\cdot \prod_{i=1}^{k}\frac{1}{1-w_i}\times \\
%\times \prod_{\sq,\sq'\in \lambda}\omega(\sq/\sq')\cdot \prod_{\sq\in \lambda}\prod_{i=1}^{k}\omega(\sq/w_i)\cdot \prod_{i<j}\frac{1-w_i/w_j}{1-qw_i/w_j}.
%\end{multline}
%\end{corollary}

By using \eqref{eq:tangent char} and \eqref{eq: arms legs identity}, one can check the following:

\begin{proposition}
%\subsection{}\todo{Do we still need this}
Let $a(\sq,j)$  denote the arm of $\sq$ in $\lambda_{n-k+j}$. Let $l(\sq)$ denote the leg of $\sq$ in $\lambda_n$.
%\begin{theorem}\todo{deprecated}
%The equivariant character of the tangent space to  $\PFH_{n,n-k}$ at a point $\lambda_{\bullet}}=(\lambda^{(n)}\supset \ldots\supset     
%\lambda^{(n-k)})$
equals
\[
\ch T_{\lambda_{\bullet}}(\PFH_{n,n-k})=kq+\sum_{\sq\in \lambda_{n-k}}\theta(\sq)
\]
where
\[
\theta(\sq)= q^{a(\sq,0)+1}t^{-l(\sq)}+q^{-a(\sq,k)}t^{l(\sq)+1}
\]
if there are no boxes in $\lambda_{n}\setminus \lambda_{n-k}$ above $\sq$, and
\[
\theta(\sq)=q^{a(\sq,i)+1}t^{-l(\sq)-1}+q^{-a(\sq,i-1)}t^{l(\sq)+1}
\]
if there is a box labeled by $i$ above $\sq$.
\end{proposition}

\section{Geometric operators}
\subsection{K-theory}
\begin{definition}
An algebraic variety $X$ with an action of $\BT=\BC^*_{q}\times \BC^{*}_t$ will be called \emph{good} if
\begin{enumerate}
\item $X$ is smooth,
\item all the $\BT$ fixed points on $X$ are isolated. 
\end{enumerate}
%\todo{Do we need existence of certain limits in $X$? AM}
\end{definition}
Let $X$ be a good space. We denote by $K(X)$ the $\BT$-equivariant K-theory of $X$ and by $\bar K(X)$ the localization
\[
\bar K(X) = K(X)\otimes_{\BQ[q^{\pm1},t^{\pm1}]} \BQ(q,t).
\]
For a fixed point $x\in X$ we denote by $[x]=O_x$ its class in $K(X)$ and by $[x]'$ the \emph{dual} class
\[
[x]' = \frac{[x]}{\Lambda^* \Omega_x} \in \bar K(X),
\]
where 
\[
\Lambda^* \Omega_x = \sum_{i} (-1)^i \Lambda^i \Omega_x.
\]
%Then we have (a useful reference is \cite{okounkov2015lectures})%\todo{Is there a better one? AM}
%\begin{proposition}\label{prop:push pull}
Let $f:X\to Y$ be an equivariant map between good spaces. The pullback map in equivariant $K$-theory is given as follows: for any fixed point $y\in Y$ we have
\[
f^* [y]' = \sum_{x\in X^\BT: f(x)=y} [x]'.
\]
If $f$ is proper, then for any fixed point $x\in X$ we have
\[
f_* [x] = [f(x)].
\]
\begin{remark}
By Thomason localization theorem we have an isomorphism 
\[
\bar K(X)\cong \bigoplus_{x\in X^{\BT}} [x] \bar K(\text{point}),
\]
see e.g. \cite{okounkov2015lectures}. Thus we can define $f_*$ by the above formula even if $f$ is not proper.
\end{remark}

%\end{proposition}

By abuse of notation we will denote by $I_{\lambda,w}\in K(\PFH_{n,n+k})$ resp. $I_{\lambda,w}'\in \bar K(\PFH_{n,n+k})$ the class resp. the dual class of the fixed point $I_{\lambda,w}$.

\subsection{Affine Hecke action}

For $1\le m\le k-1$ consider the space $\PFH_{n,n-k}^{(m)}$ consisting of partial flags $I_n\subset \ldots \subset I_{n-m+1}\subset I_{n-m-1}\subset \ldots I_{n-k}$ with the same condition $yI_{n-k}\subset I_{n}$. In complete parallel with Theorem \ref{th:smooth}, one can prove that this space is smooth.  There is a natural projection $\pi:\PFH_{n,n-k}\to \PFH_{n,n-k}^{(m)}$, which is projective. For a fixed point $I_{\lambda^{(\bullet)}}\in\PFH_{n,n-k}$ we have that $\pi(I_{\lambda^{(\bullet)}})=I_{\lambda'^{(\bullet)}}$ where the sequence of partitions $\lambda'^{(\bullet)}$ is obtained from $\lambda^{(\bullet)}$ by removing $\lambda^{(n-m)}$. There is at most one other fixed point that goes to $I_{\lambda'^{(\bullet)}}$, corresponding to a sequence which we denote by $s_m(\lambda^{(\bullet)})$. If $I_{\lambda^{(\bullet)}}$ is specified as $I_{\lambda,w}$ then $I_{s_m(\lambda^{(\bullet)})}=I_{\lambda,s_m(w)}$, where $s_m$ swaps $w_m$ and $w_{m+1}$. A formula similar to \eqref{eq:tangent char} can be proved for $I_{\lambda'^{(\bullet)}}$, we have
\[
\ch \Omega_{\lambda'^{(\bullet)}} \PFH_{n,n-k}^{(m)} = qt B_{\lambda^{(n-k)}} + B_{\lambda^{(n)}}^* - (t-1)(q-1) B_{\lambda^{(n-k)}} B_{\lambda^{(n)}}^* + (q-1) \sum_{k-1\geq i\geq j \geq 1} w_i' w_j'^*,
\]
where 
\[
w_i' = \begin{cases}
w_i & (i<m),\\
w_m + w_{m+1} & (i=m),\\
w_{i+1} & (i>m).
\end{cases}
\]
Therefore we have
\[
\ch \Omega_{\lambda'^{(\bullet)}} - \ch \Omega_{\lambda^{(\bullet)}} = (q-1) w_{m} w_{m+1}^{-1},
\]
\[
\ch \Omega_{\lambda'^{(\bullet)}} - \ch \Omega_{s(\lambda^{(\bullet)})} = (q-1) w_{m+1} w_{m}^{-1}.
\]
%Applying Proposition \ref{prop:push pull} 
We obtain
\[
\pi^* \pi_* I_{\lambda,w} = \Lambda^*\left((q-1) w_{m} w_{m+1}^{-1}\right) I_{\lambda,w} +  \Lambda^*\left((q-1) w_{m+1} w_{m}^{-1}\right) I_{\lambda,s_m(w)}
\]
\[
= \frac{1-q w_{m} w_{m+1}^{-1}}{1-w_{m} w_{m+1}^{-1}} I_{\lambda,w} + 
\frac{1-q w_{m+1} w_{m}^{-1}}{1-w_{m+1} w_{m}^{-1}} I_{\lambda,s_m(w)}.
\]
Note that the second summand should be omitted if $I_{\lambda^{(\bullet)}}$ is the only fixed point that goes to $I_{\lambda'^{(\bullet)}}$. This happens precisely when $\lambda^{(n-m+1)}\setminus\lambda^{(n-m-1)}$ is a pair of horizontally adjacent cells, i.e. $w_{m}=q w_{m+1}$. In such situation the factor in front of $I_{\lambda, s_m(w)}$ vanishes anyway, so the formula still holds formally even though $I_{\lambda, s_m(w)}$ does not correspond to a point in $\PFH_{n,n-k}$.

We get the following lemma:
\begin{lemma}
Let $T_m=\pi^*\pi_*-q.$ Then 
\begin{equation}
\label{eq: T action}
T_m(I_{\lambda,w})=\frac{(q-1) w_{m+1}}{w_{m}-w_{m+1}}I_{\lambda,w}+\frac{w_{m}-qw_{m+1}}{w_{m}-w_{m+1}}I_{\lambda, s_m(w)}.
\end{equation}
\end{lemma}

The operators $z_i$ are given by multiplication by line bundles $\CL_j=I_{n-j}/I_{n-j+1}$. Note that we have
\begin{equation}
\label{eq: z action}
\CL_j I_{\lambda,w} =  w_j I_{\lambda,w}.
\end{equation}

%%\subsection{Action of $\phi$}\todo{deprecated}

%%A more interesting family of operators act from the $K$-theory of $\PFH_{n,n+k}$ to $\PFH_{n+1,n+k+1}$.
%%There are two projections $p:\PFH_{n,n+k+1}\to \PFH_{n,n+k}$ and $q:\PFH_{n,n+k+1}\to \PFH_{n+1,n+k+1}$.
%%For any function $g$ in two variables  can construct an operator $\phi_g(\alpha)=q_{*}(p^*\alpha\otimes g(\CL_1,\CL_{k+1})).$ \todo{Both %%$p^*$, $q_*$ belong to our algebra: I think $p^*$ is $d_+$ and $q_*$ is $d_-$. Isn't it more natural to describe them instead of their %%composition $\phi$? See below. AM}
%%Now
%%$$
%%p^*H_{\lambda}=\sum_{w_k}H_{\lambda+w_k},\qquad g(\CL_1,\CL_{k+1}))\otimes p^*H_{\lambda}=\sum_{w_{k+1}}g(w_1,w_{k%%+1})H_{\lambda+w_k},
%%$$
%%and
%%$$
%%\phi_g(H_{\lambda})=\sum_{w_{k+1}} \prod_{\sq\in \lambda_n}\omega^{-1}(w_1/\sq)\cdot (1-qtw_1)\prod_{i=2}^{k+1} \frac{1-tw_1/w_i}%%{1-qtw_1/w_i}g(w_1,w_{k+1})H_{\lambda+w_{k+1}-w_1}.
%%$$

%\subsection{Action of $d_{\pm}$}

\subsection{Creation and annihilation} %\todo{Again, my take on the rest of the operators AM}
There are natural projection maps forgetting the first %\todo{depends from which side you count} 
and the last ideal respectively 
\[
f:\PFH_{n+1,n-k}\to \PFH_{n,n-k},\ g:\PFH_{n,n-k}\to \PFH_{n,n-k+1}. 
\]
Here $g$ is projective. We will denote
\[
d_- = g_*, \quad d_+ = q^k (q-1) f^*.
\]
Note that $d_+$ increases $k$ and $d_-$ decreases $k$.
\begin{lemma}\label{lem:d operators}
We have
\[
d_- I_{\lambda,w x} = I_{\lambda,w},
\]
\[
d_+ I_{\lambda,w} = -q^k\sum_{x} x d_{\lambda+x, \lambda} \prod_{i=1}^k \frac{x- t w_i}{x - qt w_i} I_{\lambda+x,x w},
\]
where $xw=(x,w_1,w_2,\ldots,w_k)$, and $d_{\lambda, \mu}$ is the Pieri coefficient
%%\[d_{\lambda,\mu}(q,t)=
%%  \prod_{s\in R_{\lambda,\mu}}
%%  \frac{q^{a_\mu(s)}-t^{l_\mu(s)+1}}
%% {q^{a_\lambda(s)}-t^{l_\lambda(s)+1}}
%%    \prod_{s\in C_{\lambda,\mu}}
%%  \frac{q^{a_\mu(s)+1}-t^{l_\mu(s)}}
%%  {q^{a_\lambda(s)+1}-t^{l_\lambda(s)}},
%%\]
%\todo{WTF are the $R_{\lambda, \mu}$ and $C_{\lambda,\mu}$??? 
% BTW a simple recipe to insert citations: go to Google Scholar, find the paper click ``cite'' choose BibTeX, copy paste, done! AM}
%\[
%d_{\mu,\mu+x}=\frac{1-q}{1-x^{-1}}\prod_{\sq\in \mu}\omega(\sq/x)
%\]
\[d_{\lambda,\mu}(q,t)=
  \prod_{s\in R_{\lambda,\mu}}
  \frac{q^{a_\mu(s)}-t^{l_\mu(s)+1}}
  {q^{a_\lambda(s)}-t^{l_\lambda(s)+1}}
    \prod_{s\in C_{\lambda,\mu}}
  \frac{q^{a_\mu(s)+1}-t^{l_\mu(s)}}
  {q^{a_\lambda(s)+1}-t^{l_\lambda(s)}}
\]
%\[d_{\lambda,\mu}(q,t)=
%  \prod_{s\in R_{\lambda,\mu}}
%  \frac{q^{a_\mu(s)}-t^{l_\mu(s)+1}}
%  {q^{a_\lambda(s)}-t^{l_\lambda(s)+1}}
%    \prod_{s\in C_{\lambda,\mu}}
%  \frac{q^{a_\mu(s)+1}-t^{l_\mu(s)}}
%  {q^{a_\lambda(s)+1}-t^{l_\lambda(s)}}
%\]
for multiplication by $e_1$ in the modified
Macdonald basis e.g. from \cite{garsia2014new} formula 3.1, which
satisfies
%which comes from the Pieri rule for the modified Macdonald polynomials:
\[
e_1 \tilde H_{\mu} = \sum_{\lambda} d_{\lambda,\mu} \tilde H_{\lambda}.
\]
%given for instance from 
Here $R_{\lambda,\mu}$ is the set of cells in
the row of the unique box in $\mu,\lambda$, and $C_{\lambda,\mu}$ is the
set of cells in the column.
%\todo{This expression does not make sense, it is $0/0$. Please put back something that makes sense + reference}

%Here 
%\[
%\omega(x)=\frac{(1-x)(1-qtx)}{(1-qx)(1-tx)}=\Lambda^*((1-q)(1-t)x).
%\]
\end{lemma}
\begin{proof}
The formula for $d_-$ is immediate from the definition. %Proposition \ref{prop:push pull}. To apply Proposition \ref{prop:push pull} 
For $d_+$ we calculate
\[
\ch \Omega_{\lambda,w} - \ch \Omega_{\lambda+x,xw} =  -x^{-1} + (t-1)(q-1) B_{\lambda^{(n-k)}} x^{-1} - (q-1) x^{-1} \sum_{i=1}^k w_i - (q-1)
\]
\[
= - x^{-1} + (t-1) (q-1) B_{\lambda} x^{-1} - (q-1) -  t(q-1)x^{-1} \sum_{i=1}^k w_i.
\]
Below we will show
\begin{equation}\label{eq:d formula}
d_{\lambda+x, \lambda} = x^{-1} \Lambda^*(- x^{-1} + (t-1) (q-1) B_{\lambda} x^{-1} + 1).
\end{equation}
Assuming \eqref{eq:d formula} we have
\[
f^* I_{\lambda,w} = \sum_{x} x d_{\lambda+x, \lambda} \frac{1}{1-q} \prod_{i=1}^k \frac{x- t w_i}{x - qt w_i} I_{\lambda+x,x w}.
\]
and we are done.

To prove \eqref{eq:d formula} we will use the following summation formula for the Pieri coefficients, see e.g. Theorem 2.4 b) in \cite{garsia2014some}:
\[
\sum_{x} d_{\lambda+x,\lambda} x^{i+1} = (-1)^i e_i[-1+(q-1)(t-1)B_\lambda] \quad(i\geq 0).
\]
Let $u$ be a formal variable. Multiplying both sides by $u^k$ and summing over $k\geq 0$ produces the following identity of rational functions:
\[
\sum_{x} d_{\lambda+x,\lambda} \frac{x}{1-ux} = \Lambda^*((-1+(q-1)(t-1)B_\lambda)u).
\]
Note that the left hand side has simple pole at $u=x^{-1}$ and 
\[
x d_{\lambda+x,\lambda} = \left((1-ux) \Lambda^*((-1+(q-1)(t-1)B_\lambda)u)\right)\big|_{u=x^{-1}}.
\]
Moving $1-ux$ inside $\Lambda^*$ we obtain
\[
x d_{\lambda+x,\lambda} = \Lambda^*((-1+(q-1)(t-1)B_\lambda)u+ux) \big|_{u=x^{-1}}.
\]
Now we can substitute $u=x^{-1}$ before applying $\Lambda^*$ and arrive at \eqref{eq:d formula}. 
\end{proof}

\begin{example}
Let $k=0$. We have $\PFH_{n,n}=\Hilb_n$. Let us identify the fixed point corresponding to a partition $\lambda$ with symmetric function 
\[
I_\lambda = \frac{\tilde H_\lambda}{\tilde H_\lambda[-1]} = (-1)^{|\lambda|} q^{-n(\lambda')} t^{-n(\lambda)} \tilde H_\lambda = \tilde H_\lambda \prod_{\square\in\lambda} (-\chi(\square)^{-1}),
\]
where $\tilde H_\lambda$ is the modified Macdonald polynomial. Then we obtain
\[
d_+ \tilde H_\lambda = - \tilde H_\lambda[-1] \sum_{w_1} w_1 d_{\lambda+w_1,\lambda} I_{\lambda+w_1, w_1},
\]
and using $\tilde H_{\lambda+w_1}[-1] = - w_1 \tilde H_\lambda[-1]$
\[
d_- d_+ \tilde H_\lambda = \sum_{w_1} d_{\lambda+w_1,\lambda} \tilde H_{\lambda+w_1},
\]
therefore $d_- d_+$ acts like the operator of multiplication by $e_1$, which matches the action of $\BA_{q,t}$ on $V_\bullet$.
\end{example}

\section{Verification of relations}
%We will now verify that these operators form a representation
%of $\mathbb{B}_{q,t}$, and therefore of $\mathbb{A}_{q,t}$ using
%the inclusion map $\beta : \mathbb{A}_{q,t} \rightarrow \mathbb{B}_{q,t}$.

Let
\[U_k=\bigoplus_{n\geq k} \bar K(\PFH_{n,n-k}),\quad
  U_\bullet= \bigoplus_{k\geq 0} U_k.\]
In this section, we will prove the following theorem:
\begin{theorem}\label{thm:verification}
  The geometric operators written as $T_i, z_i, d_+$ and $d_-$
  define a representation of the algebra $\BB_{q,t}$ on $U_\bullet$,
  and therefore a representation of $\mathbb{A}_{q,t}$ via
  the map $\beta : \mathbb{A}_{q,t} \rightarrow \mathbb{B}_{q,t}$.
\end{theorem}

We split the relations into several groups and prove them in
the subsections below.
We will denote
$H_{\lambda,w}=(-1)^{|\lambda|} q^{n(\lambda')} t^{n(\lambda)} I_{\lambda,w}$,
so that the $H_{\lambda,w}$ form a basis of $U_\bullet$.
Note that the formulas for the action of $T_m$, $\CL_j$, $d_-$ in the $H$-basis are the same as for $I$-basis.

\subsection{$z_i, T_i$} The following relations are easy to verify
\begin{proposition} The operators $z_i:=\CL_i$, $T_i$ satisfy relations of the (conjugate) affine Hecke algebra:
\begin{equation*}
(T_i-1)(T_i+q)=0, \quad T_i T_{i+1} T_i = T_{i+1} T_i T_{i+1}, \quad T_i T_j = T_j T_i \quad (|i-j|>1),
\end{equation*}
\begin{equation*}
T_i z_{i} T_i = q z_{i+1} \; (1\leq i\leq k-1),
\end{equation*}
\begin{equation*}
z_i T_j = T_j z_i  \; (i\notin \{j, j+1\}),\;
z_i z_j = z_j z_i \; (1\leq i,j \leq k),
\end{equation*}
\end{proposition}

In fact, the construction of $z_i$ and $T_i$ is very similar to the classical construction of finite-dimensional representations of the affine Hecke algebra using ``multisegments'' (see e.g. \cite{vazirani2002parameterizing}).  The operators $T_i$ and $z_i$ do not change the biggest ideal $I_n$ and the smallest ideal $I_{n-k}$. In terms of the fixed point basis, this means that we can fix two partitions $\lambda_{n-k}\subset \lambda_n$ such that the skew shape $\lambda_{n}\setminus \lambda_{n-k}$ consists of several horizontal strips. The choice of $\lambda_{n-k+1},\ldots,\lambda_{n-1}$ is equivalent to the choice of a standard tableau of this skew shape. Then \eqref{eq: T action} and \eqref{eq: z action} 
agree with the action of the affine Hecke algebra on such standard tableaux \cite{vazirani2002parameterizing,ram2003affine}. 

\subsection{$d_-, d_+, T_i$} From Lemma \ref{lem:d operators} we obtain
\[
d_+ H_{\lambda,w} = q^k\sum_{x} d_{\lambda+x, \lambda} \prod_{i=1}^k \frac{x- t w_i}{x - qt w_i} H_{\lambda+x,x w},
\]
for $wy=(w_1,\ldots,w_{k-1},y)$
\[
d_- d_+ H_{\lambda,w y} = q^k\sum_{x} d_{\lambda+x, \lambda} \frac{x-ty}{x-qty} \prod_{i=1}^{k-1} \frac{x- t w_i}{x - qt w_i} H_{\lambda+x,x w},
\]
\[
d_+ d_- H_{\lambda,w y} = q^{k-1}\sum_{x} d_{\lambda+x, \lambda} \prod_{i=1}^{k-1} \frac{x- t w_i}{x - qt w_i} H_{\lambda+x,x w},
\]
\begin{equation}
\label{comm geometry}
\frac{d_+ d_- - d_- d_+}{q-1} H_{\lambda,w y} = -q^{k-1} \sum_{x} d_{\lambda+x, \lambda} \frac{x}{x-qt y} \prod_{i=1}^{k-1} \frac{x- t w_i}{x - qt w_i} H_{\lambda+x,x w},
\end{equation}
\begin{equation}
\label{qcomm geometry}
\frac{q d_+ d_- - d_- d_+}{q-1} H_{\lambda,w y} = -q^{k} t \sum_{x} d_{\lambda+x, \lambda} \frac{y}{x-qt y} \prod_{i=1}^{k-1} \frac{x- t w_i}{x - qt w_i} H_{\lambda+x,x w}.
\end{equation}
%%Let $U_k=\bigoplus_{n\geq k} \bar K(\PFH_{n,n-k})$, $U_\bullet = \bigoplus_{k\geq 0} U_k$.
We have
\begin{proposition}
The operators $d_+, d_-, T_i$ extend to a representation of $\BA_q$ on $U_\bullet$.
\end{proposition}
\begin{proof}
The Hecke algebra relations for $T_i$ were verified above.
The relations $T_i d_- = d_- T_i$, $d_+ T_i = T_{i+1} d_+$ are straightforward. Then we need to check that
\[
d_-^2 T_k=d_-^2,\quad T_1 d_+^2 = d_+^2.
\]
The first one is straightforward. To establish the second one write
\[
d_+^2 H_{\lambda, w} = q^{2k+1} \sum_{x,y} d_{\lambda+x,\lambda} d_{\lambda+x+y,\lambda+x} \frac{y-tx}{y-qtx} \prod_{i=1}^{k-1} \frac{(x-t w_i)(y-t w_i)}{(x-q t w_i)(y-q t w_i)} H_{\lambda+x+y,yx w}.
\]
Note that there are no terms with $y=tx$. All the terms with $y=qx$ are invariant under $T_1$. Suppose $y\neq qx$, $y\neq tx$, in other words the cells $x,y$ are non-adjacent. Using \eqref{eq:d formula} we have
\[
d_{\lambda+x,\lambda} d_{\lambda+x+y,\lambda+x}
\]
\[
= (xy)^{-1} \Lambda^*\left(((q-1)(t-1)B_\lambda-1)(x^{-1} + y^{-1}) + (t-1)(q-1)x y^{-1} + 2\right),
\]
so 
\begin{equation}\label{eq:formula with C}
d_{\lambda+x,\lambda} d_{\lambda+x+y,\lambda+x} \frac{y-tx}{y-qtx} \prod_{i=1}^{k-1} \frac{(x-t w_i)(y-t w_i)}{(x-q t w_i)(y-q t w_i)} = C_{\lambda,w}(x,y) \frac{y-x}{y-qx},
\end{equation}
where the function $C_{\lambda,w}(x,y)$ is symmetric in $x,y$. So we have
\[
(T_1-1) d_+^2 H_{\lambda, w} = \sum_{x,y\,\text{non adjacent}} C_{\lambda,w}(x,y) (H_{\lambda+x+y,yxw} - H_{\lambda+x+y,xyw}) = 0.
\]

Denote by $\varphi$ the operator $\varphi=\frac{d_+ d_- - d_- d_+}{q-1}$,
\[
\varphi H_{\lambda, wy} = -q^{k-1} \sum_{x} d_{\lambda+x, \lambda} \frac{x}{x-qt y} \prod_{i=1}^{k-1} \frac{x- t w_i}{x - qt w_i} H_{\lambda+x,x w}.
\]
By Theorem \ref{th:Iwahori} it is enough to show that the following identities hold:
\[
q \varphi d_- = d_- \varphi T_{k-1},\quad T_1 \varphi d_+ = q d_+ \varphi.
\]
The first one is easier. Let 
\[
C_u = d_{\lambda+u,\lambda} \prod_{i=1}^{k-2} \frac{u-t w_i}{u-q t w_i}.
\]
Then we have
\[
q \varphi d_- H_{\lambda,w x y} = -q^{k-1} \sum_u \frac{u}{u-qtx} C_u H_{\lambda+u,u w},\quad d_- \varphi T_{k-1} 
\]
\[
 = -q^{k-1} \sum_u \left(\frac{(q-1) y}{x-y} \frac{u(u-tx)}{(u-qty)(u-qtx)} + \frac{x-qy}{x-y} \frac{u(u-ty)}{(u-qtx)(u-qty)}\right) C_u H_{\lambda+u,u w}.
\]
The rational function in parentheses equals $\frac{u}{u-qtx}$, so the identity holds. Finally we compare
\[
A=q d_+ \varphi H_{\lambda, w u} = -q^{k}\sum_{x,y} d_{\lambda+x,\lambda} d_{\lambda+x+y,\lambda+x} \frac{x(y-tx)}{(x-qtu)(y-qtx)} 
\]
\[
\times \prod_{i=1}^{k-1} \frac{(x-tw_i)(y-t w_i)}{(x-qtw_i)(y-qtw_i)} H_{\lambda+x+y,yxw}
\]
and
\[
B=T_1 \varphi d_+ H_{\lambda, w u} = -q^k T_1 \sum_{x,y} d_{\lambda+x,\lambda} d_{\lambda+x+y,\lambda+x} \frac{y(y-tx)(x-tu)}{(y-qtu)(y-qtx)(x-qtu)} 
\]
\[
\times \prod_{i=1}^{k-1} \frac{(x-tw_i)(y-t w_i)}{(x-qtw_i)(y-qtw_i)} H_{\lambda+x+y,yxw}.
\]
Similar to the computations with $d_+^2$ we analyze two cases. If $y=qx$, i.e. $x$ and $y$ are adjacent, we have $T_1 H_{\lambda+x+y,yxw}=H_{\lambda+x+y,yxw}$ and coefficients of these terms coincide. Suppose $x$ and $y$ are not adjacent. Using \eqref{eq:formula with C} we write the coefficient of $H_{\lambda+x+y,yxw}$ in $A$ as
\[
\frac{x(y-x)}{(x-qtu)(y-qx)} C_{\lambda,w}(x,y). 
\]
Using symmetry of $C_{\lambda,w}(x,y)$, we see that the corresponding coefficient in $B$ is
\[
\left(\frac{(q-1)xy(y-x)(x-tu)}{(y-x)(y-qtu)(y-qx)(x-qtu)} + \frac{(x-qy)x(x-y)(y-tu)}{(x-y)(x-qtu)(x-qy)(y-qtu)}\right)
\]
\[
\times C_{\lambda,w}(x,y).
\]
Comparing the rational functions we see that the coefficients coincide.
\end{proof}

\subsection{$d_-, d_+, z_i$}
%Finally, the connection between the $\BA_q$-structure and the conjugate affine Hecke algebra structure is expressed in the following relations that are easy to prove:
It remains to check the following relations:
%\begin{proposition}
%We have
\[
z_i d_- = d_i z_i,\quad d_+ z_i = z_{i+1} d_+,
\]
\[
z_1 (q d_+ d_- - d_- d_+) = qt (d_+ d_- - d_- d_+) z_k.
\]
%\end{proposition}
The proof of the first two is straightforward, and the last one immediately follows from \eqref{comm geometry} and \eqref{qcomm geometry}. The proof of Theorem \ref{thm:verification} is complete.

\subsection{Serre duality}
\label{sdsection}
We have two additional involutions on
$K(\PFH_{n,n+k})$ and $\bar K(\PFH_{n,n+k})$, given by Serre duality and dualization
of vector bundles, respectively:
\[
\SD \left(\sum_{\lambda,w}a_{\lambda,w}(q,t)I_{\lambda,w}\right)=\sum_{\lambda,w}a_{\lambda,w}(q^{-1},t^{-1})I_{\lambda,w},
\]
\[
\left(\sum_{\lambda,w}a_{\lambda,w}(q,t)I'_{\lambda,w}\right)^*=\sum_{\lambda,w}a_{\lambda,w}(q^{-1},t^{-1})I'_{\lambda,w}.
\]
We have another involution $\mathcal{N}=\mathcal{L} \SD \mathcal{L}^{-1}$,
where $\mathcal{L}$ is the pullback of the determinant of the tautlogical bundle
from $\Hilb_n$, satisfying $H_{\mu,w}=(-1)^{|\mu|} \mathcal{L} I_{\mu,w}$.
\begin{equation}
  \label{NUdef}
\CN\left(\sum_{\lambda,w}a_{\lambda,w}(q,t)H_{\lambda,w}\right)=\sum_{\lambda,w}a_{\lambda,w}(q^{-1},t^{-1})H_{\lambda,w}.
\end{equation}
%We will see that this is the operator that corresponds to involution defined on $V_\bullet$ introduced in \cite{carlsson2015proof}. 

This operator has the commutation relations
agreeing with \eqref{invsymmetry}, justifying calling it $\mathcal{N}$:
\begin{proposition}
  \label{invprop}
One has 
\[
\CN d_{-}\CN=d_{-},\; \CN T_i \CN=T_i^{-1},\; \CN d_{+}\CN = q^{-k} z_1 d_{+} =\beta(d_+^*).
\]
\end{proposition}

%\begin{remark}
%The last equation is different from the last equation in \eqref{invsymmetry} by a factor. This can be resolved by scaling all $z_i$
%by a factor $(-qt)$. It is easy to see that scaling will preserve all defining relations in $\BB_{q,t}$.
%\end{remark}
%
\begin{proof}
The first equation is clear from Lemma \ref{lem:d operators}. For the second, observe that the Hecke relations imply
\[
T_m^{-1}=q^{-1}T_m+q^{-1}(q-1).
\]
%% 0=(T_i-1)(T_i+q)=T_i^2+(q-1)T_i-q=T_i+(q-1)-qT_i^{-1}
On the other hand, by \eqref{eq: T action} one has
\[
\CN T_m\CN (H_{\lambda,w})=\frac{(q^{-1}-1) w^{-1}_{m+1}}{w^{-1}_{m}-w^{-1}_{m+1}}H_{\lambda,w}+\frac{w^{-1}_{m}-q^{-1}w^{-1}_{m+1}}{w^{-1}_{m}-w^{-1}_{m+1}}H_{\lambda, s_m(w)}=
\]
\[
q^{-1}\left[\frac{(q-1) w_{m}}{w_{m}-w_{m+1}}H_{\lambda,w}+\frac{w_{m}-qw_{m+1}}{w_{m}-w_{m+1}}H_{\lambda, s_m(w)}\right]=q^{-1}\left[(q-1)+T_{m}\right].
\]
Finally,
\[
\CN d_{+}\CN=q^{-k}\sum_{x} d_{\lambda+x, \lambda} (q^{-1},t^{-1})\prod_{i=1}^k \frac{x^{-1}- t^{-1} w_i^{-1}}{x ^{-1}- q^{-1}t^{-1} w_i^{-1}} H_{\lambda+x,x w}=
\]
\[
\sum_{x} x d_{\lambda+x, \lambda} \prod_{i=1}^k \frac{x- t w_i}{x - qt w_i} H_{\lambda+x,x w}=q^{-k}z_1 d_{+}.
\]
Here we used the fact that $d_{\lambda+x, \lambda} (q^{-1},t^{-1})=x d_{\lambda+x, \lambda} (q,t).$
\end{proof}

\section{Comparison with the polynomial representation}

%\subsection{Isomorphism with $V_\bullet$}

Theorem \ref{thm:verification} showed that there is an action of
$\mathbb{A}_{q,t}$ on $U_\bullet$, and so in particular an
action of the subalgebra $\mathbb{A}_q \subset \mathbb{A}_{q,t}$.
It is an immediate consequence of Proposition \ref{halfalgprop}
%\begin{proposition}
that there is a unique $\BA_q$-equivariant sequence of maps $\Phi_k:V_k\to U_k$ sending $1\in V_0$ to $H_{()}\in K(\PFH_{0,0})$.
%where the $\mathbb{A}_q$ action on $U_\bullet$ is given 
%\end{proposition}
We denote by $\Phi: V_\bullet \to U_\bullet$ the resulting map.

In this section, we will prove:
\begin{theorem}
  \label{isothm}
  The map $\Phi_k$ is an isomorphism. Moreover, we have that
  \[\Phi_0(H_{\mu})=\tilde{H}_{\mu},\]
%  where $\tilde{H}_{\mu}$ is the modified Macdonald polynomial.
  where $\tilde{H}_{\mu}$ is the modified Macdonald polynomial,
  and that $\Phi_k \CN=\CN \Phi_k$,
  where the two operators denoted $\CN$ are the
  involutions in equations \eqref{NVdef} and
  \eqref{NUdef}.
  
%  and that $\Phi_k$ intertwines the involution $\mathcal{N}$ on $U_\bullet$ with the
\end{theorem}

We now start proving this theorem, beginning with the statement that $\Phi_k$ is
an isomorphism.

Let $V_{n,k}$ denote the degree $(n-k)$ part of $V_k$. Let $U_{n,k} = \bar K(\PFH_{n,n-k})$. It is clear that the bi-degrees of $T_i, d_-, d_+$ are $(0,0)$, $(0,-1)$, $(1,1)$ respectively both in $V_{n,k}$ and $U_{n,k}$, so that $\Phi$
preserves the bi-grading.
We begin by showing that $V_{n,k}$ and $U_{n,k}$ have the same dimension.

Define two collections of sets by
\[
A(n,k)=\left\{(\mu,a)\in \mathcal{P}\times 
\mathbb{Z}_{\geq 0}^k : |\mu|+|a|=n-k\right\},
\]
\[
M(n,k)=\left\{\lambda^{(n)}\supset\cdots\supset \lambda^{(n-k)}: \lambda^{(n-i)}\in \mathcal{P}_{n-i},\; \lambda^{(n)}\setminus \lambda^{(n-k)}\mbox{ is a horizontal strip}\right\}.
\]
Then the elements of $M(n,k)$ are just the indices $\lambda^{(\bullet)}$ of the basis $H_{\lambda^{(\bullet)}}$ of $U_{n,k}$ and elements of $A(n,k)$ index elements
\begin{equation}
\label{vmuadef}
v_{\mu,a}=d_-^l y_1^{a_k}\cdots y_{k}^{a_1} y_{k+1}^{\mu_l}\cdots y_{k+l}^{\mu_1},
\end{equation}
which make up a basis of $V_{n,k}$, because the Hall-Littlewood polynomials make
up a basis of symmetric functions.
%To prove the dimension claim, it therefore suffices to show that $A(n,k)$ and $M(n,k)$ have the same size.
Define a function $A(n,k)\rightarrow M(n,k)$
by the following procedure: given $\mu, a$ we set
\[
\lambda^{(n-i)} = \sort(\mu_1, \mu_2, \ldots, \mu_{l(\mu)}, a_1, \ldots, a_i, a_{i+1}+1,\ldots,a_k+1)' \quad(0\leq i\leq k),
\]
where $\sort$ transforms a sequence into a partition by sorting the entries and throwing away zeros, and $'$ takes the conjugate partition.
For instance, we would have
\[[3,1],[1,0,1,2,3] \mapsto
[7,5,3,1],[7,4,3,1],[6,4,3,1],[6,3,3,1],[6,3,2,1],[6,3,2]].\]
It is straightforward to see that this is a bijection, proving that the
two spaces have the same dimension.

We will prove our theorem by showing that
$\Phi_k$ has a triangularity property with
respect to a partial order on $A(n,k) \leftrightarrow M(n,k)$
that we now define:
Given $(\mu,a)\in A(n,k)$, and some $l$ greater than the length
of $\mu$, let
\[\alpha=(\mu;a+1)_{l}^{rev}=(a_{k}+1,...,a_{1}+1;\mu_l,...,\mu_1)\]
denote the reversed order of the concatenation
of $\mu$ and $(a_1+1,...,a_k+1)$, which always has at least one
leading zeros included in the $\mu$ terms.
%It does not matter how many trailing zeroes are included
%in $\mu$, but for convenience, we assume that there is always at least one,
%and let $l$ denote the length of $\mu$, including these zeroes.
For instance, if we took $(\mu,a)=([2,1];(1,0,2))$, and chose $l=4$, we would have
\[\alpha=(\mu;a+1)^{rev}_4=(3,1,2,0,0,1,2).\]
We will describe the procedure for determining how to compare
two elements in terms of these vectors.

For any $(\mu,a)$, we start by asserting
the following moves 
produce an element that is larger
in this order in $A(n,k)$.
In our description, the operation ``set $\alpha_i=c$ and sort''
means to make the desired substitution, then sort the leading 
``partition terms''
if $i\leq l$, so as to obtain something that we may regard as
an element of $A(n,k)$.
In the example above, the operation ``set $\alpha_4=2$ and sort''
would yield $(3,1,2,0,1,2,2)$, corresponding to $\mu=[2,2,1]$, and
$a=(1,0,2)$.
\begin{enumerate}
\item If $\alpha_i>\alpha_j$ for $i<j$, set
  $(\alpha_i,\alpha_j)=(\alpha_j,\alpha_i)$, i.e. switch the labels and sort.
\item If $\alpha_i<\alpha_j-1$ for any $i,j$, set
  $(\alpha_i,\alpha_j)=(\alpha_j-1,\alpha_i+1)$ and sort.
\end{enumerate}
We let $\leq_{bru}$ denote the binary relation transitively
generated by these moves, which we can see does not depend on $l$, provided it is large enough.
This is in fact a partial order, which can be seen using an
alternative description in terms of the Bruhat order on affine permutations
for $GL_{k+l}$. To see this, fix some value of $l$, and
let $\widehat{W}= \mathbb{Z}^{k+l}\ltimes W_0$ denote the affine Weyl group
for $GL_{k+l}$. Now identify compositions $\alpha$ with sorted final $l$ coordinates with elements of
$S_{l} \backslash \widehat{W}/S_{k+l}$, by choosing a representative of minimal length from each coset,
of which there is a unique one.
Then $\leq_{bru}$ is the order induced by the Bruhat order on $\widehat{W}$.
Without the sorting condition from the second action of $S_l$,
this also appears in \cite{haglund2008combinatorial}.
Notice that for $k=0$ it becomes the usual dominance order on partitions.
%On the other hand,
%  on the subset in which $\mu$ is empty, it agrees with the Bruhat order on
%  compositions. We also denote by $\leq_{dom}$
%  the corresponding order on $M(n,k)$ under $\varphi_k$.

\begin{proposition}
We have that
\begin{equation}
\label{lteq}
\Phi_k(v_{\mu,a})= \sum_{(\nu,b) \leq_{bru} (\mu,a)} c_{\nu,b}(q,t) H_{\nu,b}
\end{equation}
with $c_{\mu,a}(q,t)\neq 0$.
\end{proposition}

\begin{proof}

Given 
\[f=\sum_{(a,\mu)} c_{a,\mu}(q,t) H_{\mu,a} \in U_{n,k},\] 
let $\terms(f)$ denote the set of those $(a,\mu) \in A(n,k)$
such that $c_{\mu,a}(q,t)\neq 0$.
Let us write equation \eqref{lteq} as
\[\LT(\Phi_k(v_{\mu,a}))=(\mu,a),\]
where the statement $\LT(f)=(\mu,a)$ asserts that 
$(\mu,a)\in \terms(f)$, and is greater than all other elements
with respect to $\leq_{bru}$. Note that not every $f$
has a leading term because $\leq_{bru}$ is only a partial order.

Let $b=s_i(a)$, the result of switching the labels $a_i,a_{i+1}$. Then we use the following
description of the terms of our operators:
\begin{align*}
\terms(T_{k-i}^{\pm 1}(H_{\mu,a})) = &
    \{(\mu,a),(\mu,b)\}
\\
\terms(\varphi(H_{\mu,a})) =&
\left\{(\mu\cup \{a_1+1\}-\{i\},(a_2,...,a_k,i))\right\}.
\\
\terms(d_-(H_{\mu,a})) =&  \{(\mu \cup \{a_1+1\},(a_2,...,a_k))\}.
\end{align*}
In the second to last line,
$\nu-\{i\}$
means the result of removing one of the occurences of $i$, where
$i$ ranges over all possible elements
that can be removed.
We include the case where $i$ is zero, and make the sensible convention
that $0 \in \nu$ for any $\nu$, and that $\nu-\{0\}=\nu$.

From these statements, we can check that 
\begin{equation}
\label{LTrules}  
\begin{split}
    \LT(T_{k-i}^{\pm 1}(H_{\mu,a})) & =\max\left((\mu,a),(\mu,s_i(a))\right),\\
%\LT(T^{-1}_{k-i}(H_{\mu,a})) & =\max\left((\mu,a),(\mu,s_i(a))\right),\\
    \LT(\varphi(H_{\mu,a})) & =(\mu,(a_2,...,a_{k},a_1+1)), \\
\LT(d_-(H_{\mu,a})) & =(\mu \cup \{a_1+1\},(a_2,...,a_{k})).
\end{split}
\end{equation}
It follows from the properties of the Bruhat order on
$\widehat{W}$ that if $(\mu,a)\leq_{bru} (\nu,b)$, then 
\begin{equation}
\label{LTpreserves}
\begin{split}
(\mu,s_i(a)), (\mu,a) &\leq_{bru}  (\nu,s_i(b)), 
\quad \mbox{if $b \leq_{bru} s_i(b)$},\\
(\mu,(a_2,...,a_k,a_{1}+1) &\leq_{bru} (\nu,(b_2,...,b_{k},b_1+1)), \\
(\mu \cup \{a_1+1\},(a_2,...,a_k)) &\leq_{bru} (\nu\cup \{b_1+1\},(b_2,...,b_k)).
\end{split}
\end{equation}
%\begin{equation}
%\label{LTpreserves}
%\begin{split}
%(\mu,s_i(a)) &\geq_{bru}  (\nu,s_i(b)), (\nu,b)\quad \mbox{if $s_i(a)\geq_{bru} a$}\\
%(\mu,(a_2,...,a_{k},a_1+1)) &\geq_{bru} (\nu,(b_2,...,b_k,b_{1}+1),\\
%%(\mu,(a_2,...,a_k)) &\geq_{bru} (\nu,(b_2,...,b_k)).
%\end{split}
%\end{equation}
The second set of equations gives conditions for when 
$A(f)$ has a leading term depending only on the leading term for $f$
for each operator $A$,
and the first set describes what that leading term is.
These two sets of rules will be enough to prove the result.

By the statements about $d_-$ in \eqref{LTrules} and \eqref{LTpreserves}, 
it suffices to prove the proposition in the case 
when $\mu$ is the empty partition. 
We will prove this by induction on $|a|$.
If $m=\max(a)$ is zero, then we are done.
Otherwise, let $i$ be the smallest index such that $a_i=m$. 
Let $g\in U_{n,k}$ be any element with a leading term given by
$\LT(g)=(\emptyset,b)$,
where $b$ is the composition that agrees with $a$, except that $b_i=a_{i-1}$.
It suffices to show that $\LT(y_{k-i}g)=(\emptyset,a)$, 
where $y_i$ is the operator on $U_{n,k}$
defined in terms of $T_i,T_i^{-1},\varphi$ by equation \eqref{yidef}.
Note the reversal of the ordering of $a$
in the definition \eqref{vmuadef} of the basis $v_{\mu,a}$, which is
why we use $y_{k-i}$ instead of $y_i$.

Consider the sequences of elements of $U_{n,k}$ given by 
\begin{align*}
& g^i=g, & g^{j}=T_{k-j}(g^{j+1})\mbox{ for $1\leq j \leq i-1$},\\
& f^k=\varphi(g^1), & f^{j}=T_{k-j}^{-1}(f^{j+1})\mbox{ for $i\leq j \leq k-1$}.
\end{align*}
We also define a sequence of compositions by
\[b^j=s_j(b^{j+1}), \quad a^k=\left(b^1_2,...,b^1_k,b^1_1+1\right),\quad a^{j}=s_j(a^{j+1}).\]
For instance, if $a=(2,0,3,1,3,0,3,0,1)$, then we would have $i=3$, and 
\[b^3,b^2,b^1,a^9,a^8,a^7,a^6,a^5,a^4,a^3=
(2, 0, 2, 1, 3, 0, 3, 0, 1), (2, 2, 0, 1, 3, 0, 3, 0, 1),\]
\[(2, 2, 0, 1, 3, 0, 3, 0, 1),(2, 0, 1, 3, 0, 3, 0, 1, 3),
  (2, 0, 1, 3, 0, 3, 0, 3, 1),\]
\[(2, 0, 1, 3, 0, 3, 3, 0, 1),
  (2, 0, 1, 3, 0, 3, 3, 0, 1), (2, 0, 1, 3, 3, 0, 3, 0, 1),\]
\[(2, 0, 1, 3, 3, 0, 3, 0, 1), (2, 0, 3, 1, 3, 0, 3, 0, 1).\]
By \eqref{yidef}, we have that $f=f^i$, and we clearly have that $a=a^i$.
It therefore suffices to prove the the more general statement that
\[(\emptyset,a^j)=\LT(f^j),\quad (\emptyset,b^j)=\LT(g^j)\]
for all $j$.

To see this, notice that we have $a^j \leq_{bru} a^{j-1}$, and $b^j\leq_{bru} b^{j-1}$.
The first statement follows simply because $a_i=m$ is the maximum entry,
and so the order can only be increased by moving it to the left.
The second statement follows because $i$ is the leftmost occurence
of the maximum entry, so $b_i=m-1$ greater than or equal to every term
to its left.
Therefore, the condition in the first part of \eqref{LTpreserves} is satisfied,
and the desired statement follows by induction from
the first two parts of equations \eqref{LTrules} and \eqref{LTpreserves}.

\end{proof}

To complete the proof of Theorem \ref{isothm},
we first see that $\Phi_k \CN=\CN\Phi_k$ by Proposition \ref{invprop},
so it only remains to show that the fixed points map to the
modified Macdonald polynomials for $k=0$.
%First, since we know that $\Phi_k$ is an isomorphism,
%we find that $\CN$ agrees with the involution of \cite{carlsson2015proof} by \ref{invprop}.
For $k=0$, it was proved in \cite{carlsson2015proof}
that $\CN$ acts as $\nabla$ composed with conjugation, i.e.
\[
\CN\left(\sum_{\lambda,w}a_{\lambda,w}(q,t)\tilde{H}_{\lambda}\right)=\sum_{\lambda,w}a_{\lambda,w}(q^{-1},t^{-1})\tilde{H}_{\lambda}.
\]
In \cite{garsia1996remarkable}, it was shown that the ring of symmetric functions are generated by the multiplication operator $e_1$,
and $\nabla e_1 \nabla^{-1}$, or equivalently, $\CN e_1 \CN$. It therefore suffices to show that $\CN,e_1$ have the
same representation in each basis. The involution $\CN$ fixes both sets of basis by definition. To show that $e_1$
has the same coefficients, it suffices to notice that $e_1=d_- d_+$ when restricted to $V_0$,
and recall that the coefficients in Lemma \ref{lem:d operators} are just the coefficients in the Pieri rule for $e_1$.
\qed

\section{Examples}

\subsection{Simple Nakajima correspondences}

An important collection of operators on the $K$-theory of Hilbert schemes can be defined as follows. Consider {\em nested Hilbert scheme}
$\Hilb^{n,n+1}=\{J\subset I\subset \BC[x,y]\}$, where $J$ and $I$ are ideals of codimensions $(n+1)$ and $n$, respectively. The variety $\Hilb^{n,n+1}$ is well known to be smooth \cite{ellingsrud1999cobordism} and carries a natural line bundle $\CL:=I/J$.
It has two projections $f:\Hilb^{n,n+1}\to \Hilb^n$ and $g:\Hilb^{n,n+1}\to \Hilb^{n+1}$ which send a pair $(J\subset I)$ to $I$ and $J$, respectively. In the constructions of \cite{FT,SV} the crucial  role was played by the operators 
\[
P_{1,k}:K(\Hilb^n)\to K(\Hilb^{n+1}),\ P_{1,k}:=g_*(\CL^{k}\otimes f^*(-)). 
\]
Remark that the quotient $I/J$ in the nested Hilbert scheme is supported at one point, which can be translated to the line $\{y=0\}$.
Thus, $\Hilb^{n,n+1}=\PFH_{n+1,n}\times \BC_t$, and $K(\Hilb^{n,n+1})\subset U_1$. Using the algebra $\BA_{q,t}$, we can realize these operators as a composition of three:
\[
(q-1) f^*=d_+:U_0\to U_1,\quad \CL_k=z_1^k:U_1\to U_1,\quad g_*=d_-:U_1\to U_0,
\]
so
\[
P_{1,k}=\frac{1}{(q-1)(1-t)}d_{-}z_1^k d_{+}.
\]

\subsection{Generators of the elliptic Hall algebra}

We will need an explicit formula for the action of $y_1$ on $U_1$. Since there are no $T$'s and $k=1$, by \eqref{comm geometry} we have
\begin{equation}
\label{y1 level 1}
y_1(H_{\lambda,y})=\frac{1}{q-1}[d_+,d_-]H_{\lambda,y}=- \sum_{x} d_{\lambda+x, \lambda} \frac{x}{x-qt y}  H_{\lambda+x,x }.
\end{equation}
This immediately implies the following result:
\begin{proposition}
The following identity holds for all $S_i$:
\begin{equation}
\label{level 1 product}
d_{-}(z_1^{S_n}y_1z_1^{S_{n-1}}y_1\cdots z_1^{S_1}y_1)d_{+}=(-1)^n\sum_{T} \prod_{i<j}\omega(w_i/w_j) \frac{w_i^{S_i+1}}{w_i-qtw_{i-1}}H_{\lambda},
\end{equation}
where $T$ is a standard tableaux of shape $\lambda$ and size $n$, $w_i$ is the $q,t$-content  of the box labeled by $i$ in $T$, and 
\[
\omega(x)=\frac{(1-x)(1-qtx)}{(1-qx)(1-tx)}.
\]
\end{proposition}

\begin{proof}
It is sufficient to prove by induction that 
\[
(z_1^{S_n}y_1z_1^{S_{n-1}}y_1\cdots z_1^{S_1}y_1)d_{+}=(-1)^n\sum_{T} \prod_{i<j}\omega(w_i/w_j) \frac{w_i^{S_i+1}}{w_i-qtw_{i-1}}H_{\lambda,\sq_n}.
\]
If we apply $y_1$ to the right hand side, we need to sum over all possible ways to add a box $w_{n+1}$ to a standard Young tableau $T$, that is, over all standard Young tableaux of size $(n+1)$. The additional factor is described by \eqref{y1 level 1} with $x=w_{n+1}$ and $y=w_n$:
\[
-d_{\lambda+w_{n+1}, \lambda} \frac{w_{n+1}}{w_{n+1}-qt w_n}=-\prod_{i\le n}\omega(w_i/w_{n+1})\frac{w_{n+1}}{w_{n+1}-qt w_n}.
\]
The action of $z_1^{S_{n+1}}$ on the result just adds a factor $w_{n+1}^{S_{n+1}}$.
\end{proof}

As a corollary, we obtain a different proof of the formula from \cite{negut2012moduli} for the generator $P_{m,n}$ of the elliptic Hall algebra
(for coprime $m$ and $n$). Indeed, it was proved in  \cite{mellit2016toric} that 
\[
P_{m,n}=d_{-}(z_1^{S_n}y_1z_1^{S_{n-1}}y_1\cdots z_1^{S_1}y_1)d_{+},\ S_i=\left\lfloor\frac{mi}{n}\right\rfloor-\left\lfloor\frac{m(i-1)}{n}\right\rfloor.
\]
By substituting these values of $S_i$ into \eqref{level 1 product} we obtain the desired formula. 

\bibliographystyle{amsalpha}
\bibliography{refs}

\end{document}